\documentclass{amsart}
\usepackage{amsthm}
\usepackage{amssymb, latexsym}
\usepackage{enumerate}
\usepackage{mathtools,color}
\usepackage{dsfont}
\usepackage{bbm}
\usepackage{ulem}

\usepackage{hyperref}

\usepackage{bbold}

\newcommand\dive{\operatorname{div}}

\newcommand{\m}[1]{\mathbb{#1}}

\newcommand{\q}[1]{\mathcal{#1}}

\newtheorem{prop}{Proposition}
\newtheorem{cor}{Corollary}
\newtheorem{thm}[prop]{Theorem}
\newtheorem{rem}{Remark}
\newtheorem{lem}[prop]{Lemma}

\numberwithin{equation}{section}

\begin{document}

\title[Non-scaling invariant semilinear wave equations] {The blow-up rate for a loglog non-scaling invariant semilinear wave equation}

\author[Tristan Roy and Hatem Zaag]{{Tristan Roy}\\
 \textit{\tiny American University of Beirut, Department of Mathematics}\\
{Hatem Zaag}\\
{\it \tiny
  Universit\'e Sorbonne Paris Nord,  LAGA, CNRS (UMR 7539)},\\
 {\it \tiny F-93430, Villetaneuse, France}
}

\address{TR: American University of Beirut, Department of Mathematics.
HZ: Universit\'e Sorbonne Paris Nord, Institut Galil\'ee,
Laboratoire Analyse, G\'eom\'etrie et Applications, CNRS UMR 7539,
99 avenue J.B. Cl\'ement, 93430 Villetaneuse, France.
}
\email{tr14@aub.edu.lb, hatem.zaag@univ-paris13.fr}

\begin{abstract}

  We consider blow-up solutions of a semilinear wave equation with a $\log\log$ perturbation of the power
  nonlinearity in the subconformal case, and show that the blow-up rate is given by
  the solution of the associated ODE which has the same blow-up time. In fact, our result shows an
  upper bound and a lower bound of the blow-up rate, both proportional to the blow-up solution of the
  associated ODE. As for the upper bound, the main difficulty comes from the fact that the PDE is not
  scaling invariant.  Our argument is a delicate adaptation of the
  original approach introduced by Hamza and Zaag for the $\log$
  perturbation of the pure power nonlinearity (see
  \cite{HZjmaa20}). It mainly relies on the design of a Lyapunov
  functional and the use of a  blow-up criterion in the similarity
  variables' setting, together with some energy, nonlinear, and interpolation estimates. As for the proof of the lower bound, we consider our argument as a major contribution, since the argument for the pure power case cannot be adapted when the scale invariance breaks
  down.
\end{abstract}

\maketitle


\section{Introduction}

Many papers are devoted to the study of finite time blow-up for
semilinear wave equations of the type
\begin{equation}\label{equ}
\partial_{tt} u -  \Delta u = f(u)
\end{equation}
considered in the whole space $\m R^N$,
where $f:\m R \longrightarrow \m R$ is regular and chosen so that the
associated ODE $u''=f(u)$ exhibits blow-up. The simplest example is
certainly the case where
\begin{equation}\label{fup}
  f(u)=|u|^{p-1}u \mbox{ with }p>1,
\end{equation}
  considered by
many authors (see John \cite{Jmm79},...).

\medskip

The Cauchy problem for \eqref{equ}-\eqref{fup} can be solved in
$H^1\times L^2(\m R^N)$, locally in time, at least for Sobolev subcritical $p$ (see for
example Shatah and Struwe \cite{SSnyu98}). If $u$ is a maximal
solution whose existence time interval is finite, following Alinhac
\cite{Apndeta95} (see also Appendix \ref{Sec:InflDomain}),  we may use the finite speed of propagation to
extend the domain of definition of $u$ to the set
\begin{equation}\label{defD}
\q D= \{(x,t)\;|\;0\le t<T(x)\},
\end{equation}
for some $1$-Lipschitz function $x\longmapsto T(x)$ with the following
property: the $H^1\times L^2$ average of the solution on sections of
any light cone with vertex $(x,T(x))$ blows up as the section
approaches the vertex. The boundary $\{t=T(x)\}$ of $\q D$ is called
the ``blow-up surface''.
In other words, the solution $u$ enjoys a ``local'' blow-up time $T(x)$, near
each $x\in \m R^N$. Note that the $1$-Lipschitz property of $T$
implies that $u$ is defined
on the backward light cone $\q C_{x,1}$ for any $x\in \m R^N$, where for any $\delta\ge
0$, $\q C_{x,\delta}$ is the backward cone with vertex $(x, T(x))$ and
slope $\delta$, formally defined by
\[
\q C_{x,\delta}=\{(\xi,\tau)\;|\; 0\le \tau < T(x) -
\delta |\xi-x|\}.
\]
If ever $x\in \m R^N$ is such that $u$ is defined on $\q C_{x,\delta}$
for some $\delta<1$, then $x$ and  $\q C_{x,\delta}$ are called
``non-characteristic''. Otherwise, we call them ``characteristic''. The set of all
non-characteristic points is denoted by $\q R$ and its complementary
by $\q S$.

\medskip


\medskip

Following these definitions, it is natural to ask whether one can derive
the ``local blow-up rate''. Such a question was solved by Merle and
Zaag in \cite{MZajm03} and \cite{MZimrn05}, in the subconformal case, where
\begin{equation}\label{condp}
p<\frac{N+3}{N-1}\mbox{ if } N\ge 2.
\end{equation}
Indeed, in \cite{MZimrn05}, given $x_0\in \q R$, the
authors proved that for any $t\in [0,T(x_0))$,
    \begin{align}
  k (T(x_0)-t)^{-\frac 2{p-1}}
  \leq& \frac{\| u(t) \|_{L^{2} \left( B \left( x_{0}, T(x_{0}) - t \right) \right) } }
{ \left(  T(x_{0}) - t \right)^{\frac{N}{2}}}\nonumber\\
+&(T(x_0)-t)
\frac{ \| \nabla u(t) \|_{L^{2} \left( B(x_{0},T(x_{0})-t) \right)} +  \| \partial_{t} u(t) \|_{L^{2} \left( B \left( x_{0}, T(x_{0}) - t \right) \right)}} {\left(  T(x_{0}) - t \right)^{\frac{N}{2}}}\nonumber\\
      \leq \; & K (T(x_0)-t)^{-\frac 2{p-1}} ,\label{2bounds}
    \end{align}
for some $0<k(N,p)\le K$, where $K$ depends only on an upper bound on
$T(x_0)$, $1/T(x_0)$, the
local initial norm $\| \left( u
(0), \partial_{t} u (0) \right) \|_{H^{1} \times L^{2}  \left( B \left(
      x_{0},  \frac{T(x_{0})}{\delta_{0}(x_{0})} \right)\right)}$
and $\delta_0(x_0)<1$, the slope of some non-characteristic
cone $\q C_{x_0,\delta_0(x_0)}\subset \q D$. Note that up to some factor, both the lower
and the upper bounds show the solution of the associated ODE
$u''=|u|^{p-1}u$ which blows up at time $T(x_0)$. The proof is given
in the framework of similarity variables first introduced in Antonini
and Merle \cite{AMimrn01} and defined for any $x_0\in \m
R^N$ and $T_0\in[0,T(x_0))$ by
\begin{equation}\label{defw}
\begin{array}{l}
y = \frac{x-x_{0}}{T_{0} - t}, \; s= - \log(T_{0} - t), \; \text{and}
  \;
  u(x,t) = (T_{0}-t)^{\frac 2{p-1}} w_{x_{0},T_{0}} (y,s) \cdot
\end{array}
\end{equation}
Given that $u$ is defined (at least) in the backward light-cone $\q
C_{x_0,T_0}$, as we have just mentioned, it follows that $w_{x_0,T_0}$
(or $w$ for simplicity) satisfies the following equation for all $s\ge
-\log T_0$ and $y\in B(0,1)$:
\begin{equation}\label{eqw}
\begin{array}{ll}
\partial_{s}^{2} w  & =  \frac{1}{\rho} \dive \left( \rho \nabla w - \rho (y \cdot \nabla w) y \right)
 - \frac{2(p+1)}{(p-1)^{2}} w +|w|^{p-1}w- \frac{p+3}{p-1} \partial_{s} w
 - 2 y \cdot \nabla \partial_{s} w
\end{array}
\end{equation}
where
\begin{equation}\label{Eqn:ValAlpha}
\rho(y):= (1 - |y|^{2})^{\alpha}\mbox{ with }\alpha := \frac{2}{p-1} - \frac{N-1}{2} > 0.
\end{equation}
Again from \cite{AMimrn01}, we know that equation \eqref{eqw} has the
following Lyapunov functional
\begin{align}
E(w(s)) :=
\int_{B(0,1)} &\left( \frac{1}{2} (\partial_{s} w(s) )^{2} + \frac{1}{2} \left( |\nabla w(s)|^{2} - |y \cdot \nabla w(s)|^{2} \right)
  +  \frac{p+1}{(p-1)^{2}} w^{2}\right. \label{defE}\\
              &\left. - |w|^{p+1}\right) \rho \; dy.\nonumber
\end{align}
This functional, combined to a blow-up criterion in the similarity
variables' setting, makes the core of the argument of \cite{MZajm03}
and \cite{MZimrn05}, together with some involved energy and interpolation estimates.

\medskip

Further refinements (both in the non-characteristic and
the the characteristic cases) or generalizations (in particular to the conformal and
superconformal cases) are available in
\cite{MZma05}, \cite{MZjfa07}, \cite{MZcmp08}, \cite{MZajm12},
\cite{MZdmj12}, \cite{HZdcds13}, \cite{CZcpam13}, \cite{MZtams16}, \cite{MZcmp15} and
\cite{MZcpam18} (see also the notes \cite{MZsnp09},  \cite{MZxedp10}, \cite{MZsls17}).
Since our focus in this paper is the blow-up rate in the
non-characteristic case for general nonlinearities in \eqref{equ}, we
won't mention those developments.


\medskip

Following our result in the case \eqref{fup}, we wonder whether the
same holds for other nonlinearities, namely that the blow-up rate near
non-characteristic points is given by the solution of the associated
ODE
\begin{equation}\label{ode}
u"=f(u).
\end{equation}
Given that the Lypunov functional $E(w)$ \eqref{defE} is central in
the argument, and that its existence seems natural, from the fact that
equation \eqref{eqw} is autonomous, we thought for a while that
extending the result to other nonlinearities would be impossible,
unless the similarity variables' equation is autonomous.

\medskip

That belief was proven to be wrong thanks to the work of Hamza and
Zaag in \cite{HZnonl12} together with Hamza and Saidi \cite{HSjdde14},
devoted to the case where
\begin{equation}\label{pert1}
  f(u) = |u|^{p-1}u+g(u)
\end{equation}
with $p$ subconformal in the sense \eqref{condp} and $g$ $C^1$ satisfying
\begin{equation}\label{condg}
  |g(u)|\le M\left(1+\frac{|u|^p}{[\log(2+|u|^2)]^a}\right)\mbox{ for some }a\ge 1.
\end{equation}
Indeed, the statement \eqref{2bounds} holds in this case, and even for more
general nonlinearities involving $x$, $t$ and also $\nabla u$ or $\partial_t u$, with a
sublinear growth, despite the fact that the similarity variables'
equation, which looks as a perturbation of \eqref{eqw}, is not
autonomous, as one may easily check. In fact, the proof in the
case \eqref{pert1} follows from a delicate and involved
adaptation of the argument of \cite{MZajm03} given in the pure power case
\eqref{fup}. It relies in particular on the design of a new Lyapunov
functional, which appears to be a clever perturbation of \eqref{defE},
in order to take into account the time dependent terms one gets in
equation \eqref{eqw}, due to the additional term in \eqref{pert1}.

\medskip

Following \cite{HZnonl12} and \cite{HSjdde14} devoted to the case
\eqref{pert1}, we still had the impression that we were handling a case
where the similarity variables' equation is asymptotically autonomous,
as one may easily check, since the main term of the nonlinearity at
infinity is a pure power. In consequence, the next challenge was to
handle a nonlinear term where the main term at infinity is not a pure power.





\medskip

There comes the recent contribution of Hamza and Zaag in
\cite{HZjmaa20} and \cite{HZna21}, where the authors consider the case
\[
f(u)=|u|^{p-1}u[\log(2+u^2)]^a\mbox{ with }p\mbox{ satisfying
\eqref{condp} and }a\in \m R.
\]
In that case, the authors prove an analogous statement to
\eqref{2bounds}, where the factor $(T(x_0)-t)^{-\frac 2{p-1}}$ is
replaced by $\psi(T(x_0)-t)$, and $\psi$ is a solution of the
associated ODE \eqref{ode} blowing up at time $t=0$. Although the
proof follows the strategy of the pure power case \eqref{fup}, the
adaptation is not easy to guess, in the sense that the similarity
variables' definition has to be adapted, and that the new Lyapunov
functional is not a perturbation of \eqref{defE}.


\bigskip

In this paper, we further investigate the case where
the nonlinearity is not a pure power, and focus on the following
``$\log \log$'' perturbation of \eqref{fup}, where
\begin{equation}\label{pert2}
f(u) =  |u|^{p-1} u g(u) \mbox{ with } g(u) := \log^{a} \left( \log
  \left( 10 + u^{2} \right) \right)\mbox{ and } a \in \mathbb{R}.
\end{equation}
Note that $\log \log$ nonlinearities of this style have already been
considered in the literature, in particular by Roy in \cite{Rapde09}.
%
%

\medskip

The local-in-time Cauchy problem for equation
\eqref{equ}-\eqref{pert2} 
is essentially well-known. 
For the reader's convenience
we write a proof in Appendix \ref{Sec:LwpH1L2}. The existence of blow-up
solutions follows from some cut-off argument applied to the solution
of the ODE \eqref{ode}. One may also use the concavity argument of
Levine \cite{Ltams74} to derive such solutions (see also Levine and
Todorova \cite{LTpams01}). Given a
blow-up solution, one may extend its domain of definition to a set $\q D$
\eqref{defD},  for some $1$-Lipschitz function $x\longmapsto T(x)$, as
for the pure-power case \eqref{fup}. The notion of characteristic and
non-characteristic points are introduced similarly.

\medskip

On the contrary, we take a different definition of similarity
variables, for any $x_{0} \in \mathbb{R}^N$ and $ T_0\in (0,T(x_{0})] $:
%
%
%
%
\begin{equation}
\begin{array}{l}
y = \frac{x-x_{0}}{T_{0} - t}, \; s= - \log(T_{0} - t), \; \text{and} \; u(x,t) = \psi_{T_{0}}(t) w_{x_{0},T_{0}} (y,s) \cdot
\end{array}
\label{Eqn:SimilVar}
\end{equation}
Here $ \psi_{T_{0}}(t) := (T_{0} - t)^{-\frac{2}{p-1}}
\log^{-\frac{a}{p-1}} \left( - \log(T_{0}-t) \right) $, is the main
term of the solution of the ODE \eqref{ode} which blows up at time
$T_0$ (see Lemma \ref{Lem:OdeAsymp}). Using equation
\eqref{equ}-\eqref{pert2}, we see
%
that $w_{x_{0},T_{0}}$ satisfies the following equation:\footnote{We write $w$ for $w_{x_{0},T_{0}}$ for sake of simplicity.}

\begin{equation}
\begin{array}{ll}
\partial_{s}^{2} w  & =  \frac{1}{\rho} \dive \left( \rho \nabla w - \rho (y \cdot \nabla w) y \right) + \frac{2a}{(p-1) s \log(s)} y \cdot \nabla w  \\
& - \frac{2(p+1)}{(p-1)^{2}} w + \gamma(s) w - \left( \frac{p+3}{p-1} - \frac{2a}{(p-1)s \log(s)} \right) \partial_{s} w
 - 2 y \cdot \nabla \partial_{s} w \\
& + e^{-\frac{2sp}{p-1}} \log^{\frac{a}{p-1}}(s) f(\phi w),
\end{array}
\label{Eqn:Wavew}
\end{equation}
where $\rho(y)$ is defined in \eqref{Eqn:ValAlpha},
%
%
\begin{align}
\gamma(s) &:= \frac{a(p+3)}{(p-1)^{2}s \log(s)} - \frac{a(a+p-1)}{(p-1)^{2} s^{2} \log^{2}(s)} - \frac{a}{(p-1) \log(s) s^{2}}   , \; \text{and} \label{Eqn:ValGamma}    \\
\nonumber \\
\phi(s) &:= e^{\frac{2s}{p-1}} (\log(s))^{-\frac{a}{p-1}} \cdot \label{Eqn:Defphi}
\end{align}
We also introduce the functional
$\mathcal{E}(w) : s \rightarrow \mathcal{E} (w(s)) $, which will serve
as a basis to derive a Lyapunov functional for equation
\eqref{Eqn:Wavew} (see below \eqref{Eqn:DefHN} together with Lemma
\ref{lemlyap}):
\begin{align}
\mathcal{E}(w(s)):=
\int_{B(O,1)} &\left( \frac{1}{2} (\partial_{s} w(s) )^{2} + \frac{1}{2} \left( |\nabla w(s)|^{2} - |y \cdot \nabla w(s)|^{2} \right)
  +  \frac{p+1}{(p-1)^{2}} w^{2}\right. \label{Eqn:Energyw}\\
 &\left. - e^{-\frac{2(p+1)s}{p-1}} (\log(s))^{\frac{2a}{p-1}} F(\phi w(s))  \right) \rho \; dy,\nonumber
\end{align}
where
\begin{equation}
\begin{array}{l}
F(x) := \int_{0}^{x} f(x') \; dx' \cdot
\end{array}
\nonumber
\end{equation}
%
%
%
%
%
%
Thanks to a delicate adaptation of the previous literature, we are
able to confirm that the blow-up rate in the non-characteristic case
is given by the solution of the associated ODE \eqref{ode}. More
precisely, this is our main statement:
%
%
\begin{thm}\label{Thm:BlowUpRate}{\rm (The blow-up rate for a $\log\log$ perturbation of the
  power nonlineairty).}
  Let $u$ be a solution of equation \eqref{equ}-\eqref{pert2} defined
  on some set $\q D$ \eqref{defD} for some $1$-Lipschitz function
  $x\longmapsto T(x)$.
  Let $x_{0}$ be a non-characteristic
point. Then there exists  $t_{0}(x_{0}) \in [ 0, T(x_{0}) ) $ such that
for all $ t \in \left[ t_0(x_{0}), T(x_{0}) \right)$,
\begin{equation}
\begin{array}{rl}
k  \leq &\frac{1}{ \psi_{T(x_{0})}(t)}  \frac{\| u(t) \|_{L^{2} \left( B \left( x_{0}, T(x_{0}) - t \right) \right) } }
{ \left(  T(x_{0}) - t \right)^{\frac{N}{2}}}\\
&+ \frac{1} {\psi_{T(x_{0})}(t)}
\frac{ \| \nabla u(t) \|_{L^{2} \left( B(x_{0},T(x_{0})-t) \right)} +  \| \partial_{t} u(t) \|_{L^{2} \left( B \left( x_{0}, T(x_{0}) - t \right) \right)}} {\left(  T(x_{0}) - t \right)^{\frac{N}{2} -1}}
\leq  K,
\end{array}
\label{Eqn:BlowUpRate}
\end{equation}
for some $0< k(N,p,a)\le K$, where $K$ depends only on an upper bound
on $T(x_0)$, $1/T(x_0)$, the norm $\| \left( u
(t_0(x_0)), \partial_{t} u (t_0(x_0)) \right) \|_{H^{1} \times L^{2}  \left( B \left(
      x_{0},  \frac{T(x_{0})-t_0(x_0))}{\delta_{0}(x_{0})} \right)\right)}$
and $\delta_0(x_0)<1$, the slope of some non-characteristic
cone $\q C_{x_0,\delta_0(x_0)}\subset \q D$.


\end{thm}

In order to find the upper bound of (\ref{Eqn:BlowUpRate}), we first find a bound on average on $s$-unit intervals $[s,s+1]$ of
$ \left\|  \left( w(\tau), \partial_{s} w(\tau \right)) \right\|^{2}_{H^{1}(B(O,1)) \times L^{2} (B(O,1))}$. More precisely, we prove in Section \ref{Sec:DesLyap} the proposition below:
\begin{prop}[Bounds on time averages of $(w,\partial_sw)$ in
  $H^1\times L^2(B(O,1))$]\label{Prop:BoundPolynws}
  Let $u$ be a solution of \eqref{equ}-\eqref{pert2}  defined
  on some set $\q D$ \eqref{defD} for some $1$-Lipschitz function
  $x\longmapsto T(x)$.
Then there  exists $b> 0$ such that for any $x_{0}$ non-characteristic point there exists $t_{0}(x_{0}) \in [ 0, T(x_{0}) ) $ such that
for all $ T_{0} \in \left[  \frac{ t_{0}(x_{0}) + T(x_{0}) }{2}, T(x_{0}) \right) $, $ x \in \mathbb{R}^N: | x - x_{0} | \leq
\frac{T_{0} - \frac{ t_{0}(x_{0}) + T(x_{0}) }{2}}{\delta(x_{0})} $, and  $ s \geq - \log \left( T^{*}(x) - t_{0}(x_{0}) \right) $  we have

\begin{equation}
\begin{array}{l}
\int_{s}^{s+1} \int_{B(O,1)} \left( ( \partial_{\tau} w )^{2} + \left| \nabla w \right|^{2} + w^{2} \right) \; dy d \tau
\leq K s \log^{1+b}(s) \;
\end{array}
\label{Eqn:EstKinetw}
\end{equation}
where $ w :=  w_{x,T^{*}(x)}$ is the function defined in
(\ref{Eqn:SimilVar}) with $ T^{*}(x) := T_{0} - \delta_{0}(x_{0})
|x-x_{0}| $, $K$ depends only on an upper bound
on $T(x_0)$, $1/T(x_0)$, the local norm \\
$\| u (t_0(x_0)), \partial_{t} u (t_0(x_0)) \|_{H^{1} \times L^{2}  \left( B \left(
      x_{0},  \frac{T(x_{0})-t_0(x_0)}{\delta_{0}(x_{0})} \right)\right)}$,
and $\delta_0(x_0)<1$, the slope of some non-characteristic
cone $\q C_{x_0,\delta_0(x_0)}\subset \q D$.
\end{prop}


\medskip

We then prove pointwise bounds of $ \left\|  \left( w(s), \partial_{s} w(s) \right) \right\|^{2}_{H^{1}(B(O,1)) \times L^{2} (B(O,1))}$. More precisely, we prove in Section \ref{Sec:ProofPropBoundFinalWs}
the following proposition:

\begin{prop}[Bounds on $(w,\partial_sw)$ in $H^1\times L^2(B(O,1))$] \label{Prop:BoundFinalWs}
  Let $u$ be a solution of \eqref{equ}-\eqref{pert2}  defined
  on some set $\q D$ \eqref{defD} for some $1$-Lipschitz function
  $x\longmapsto T(x)$.  Let $x_{0}$ be a non-characteristic
  point. Then, there exists $t_{0}(x_{0}) \in [ 0, T(x_{0}) ) $ such
  that for all $T_{0} \in \left[ \frac{t_{0}(x_{0}) + T(x_{0})}{2},
  T(x_{0}) \right) $,  $ x \in \mathbb{R}^N: | x - x_{0} | \leq
\frac{T_{0} - \frac{ t_{0}(x_{0}) + T(x_{0}) }{2}}{\delta(x_{0})} $,
  and $  s \geq - \log
  \left( T^{*}(x) - t_{0}(x_{0})
  \right) $, we have

\begin{equation}
\begin{array}{l}
\| w(s) \|_{H^{1} \left( B(O,1) \right) }^{2} + \left\| \partial_{s} w
  (s) \right\|_{L^{2} \left( B(O,1) \right)}^{2} \le K,
\end{array}
\label{Eqn:BoundFinalWs}
\end{equation}
where $ w :=  w_{x, T^{*}(x)}$ is the function defined in
(\ref{Eqn:SimilVar}), $K$ depends only on an upper bound
on $T(x_0)$, $1/T(x_0)$, the local norm
$\| u(t_{0}), \partial_{t} u (t_{0}) \|_{H^{1} \times L^{2}  \left( B \left( x_{0},  \frac{T_{0} - t_{0}(x_{0})}{\delta_{0}(x_{0})} \right)\right)}$,
and $\delta_0(x_0)<1$, the slope of some non-characteristic
cone $\q C_{x_0,\delta_0(x_0)}\subset \q D$.

\medskip


\medskip

%

\end{prop}
Proposition \ref{Prop:BoundFinalWs} yields the upper bound in Theorem
\ref{Thm:BlowUpRate} by getting back to the $x$ and $t$ variables from
the $y$ and $s$ variables thanks to the similarity variables'
definition \eqref{defw} and using the monotone convergence theorem. The proof uses intermediate results that we prove in Section
\ref{Sec:IntermResults}.

We now discuss the main interests of this paper. \\
A first interest lies in the proof of the upper bound of the estimate
(\ref{Eqn:BlowUpRate}). It is well-known that the function $loglog$ goes more slowly to infinity than the function $log$ as the argument goes to infinity. Therefore it is tempting to believe that the $loglog$ perturbation is less severe than the $log$ one; consequently, the proof of the blow-up rate should be easier. The arguments in this paper show that the rate of variation of the perturbation of  (\ref{fup}) does not play a role in proving the ``$\leq $'' inequality of (\ref{Eqn:BlowUpRate}). We design a Lyapunov functional in similarity variables adapted to the problem (\ref{pert2}) that satisfies some nice decreasing and positivity properties (see Section \ref{Sec:DesLyap}); then, by using delicate energy, interpolation and nonlinear estimates, we manage to prove polynomials bounds of $w$, a solution of (\ref{Eqn:Wavew}), and then a nice decaying property of a new Lyapunov functional (see Section \ref{Sec:IntermResults}). Finally we prove  Proposition \ref{Prop:BoundFinalWs} in Section \ref{Sec:ProofPropBoundFinalWs}. \\
A second interest lies in the proof of the lower bound of the estimate
(\ref{Eqn:BlowUpRate}). We highlight the fact that unlike the pure power case treated by Merle and Zaag in \cite{MZajm03}, the lower bound is not a direct consequence of the local-in-time well-posedness in the energy space combined with the invariance by scaling. 
In fact, due to the lack of that invariance, we had to introduce a new argument  below in Section \ref{Sec:LowerBoundProof}, which makes a major novelty of our paper. Our argument relies on a delicate analysis
of the local behavior  
of solutions of \eqref{equ}-\eqref{pert2} that takes into account the size of the scaling parameter $\lambda $ together with a modification of the scaling argument. We would like to mention that our method is valid not only for the $\log \log$ nonlinearity. It extends in fact to other types of nonlinearities, including the $\log$ nonlinearity considered by Hamza and Zaag in \cite{HZna21}, where the lower bound was given with no details of the proof. \\
Before concluding the introduction, we would like to mention that with the loglog nonlinearity, our paper is part of the recent wave of interest in non-scale-invariant PDEs. That wave includes the already mentioned papers of Hamza and Zaag on the wave equation with a log nonlinearity (see \cite{HZjmaa20} and \cite{HZna21}), together with their twin work in the parabolic case (see \cite{HZarma22}). We also mention other contributions by Souplet and co-authors, dedicated to more general non-homogeneous nonlinearities, both in the elliptic and the parabolic case. See in particular Souplet \cite{Sdcds23}, Quittner and Souplet \cite{QSarxiv24}, Chabi and Souplet \cite{CSarxiv24}.

\textbf{Acknowledgements}: This research of the first author was partially funded by a URB grant (ID: 3245), a CEDRE grant (Ref: PHC CEDRE 2021: Projet $N^{o}$ 44456XJ), plus a CAMS grant from the American University of Beirut. The first author author is grateful for the support of the Center for Advanced Mathematical Sciences (CAMS ORCID: 0009-0004-5763-5004) at the American University of Beirut. The second author wishes to thank Pierre Raphael and the ERC project SWAT for their support.

\section{Notation}
We recall some notation, some of which was already used in the
introduction. If $X\in \m R$,  then $\langle X \rangle =
\left(1+X^2\right)^{\frac 12}$.\\
We write $X\ll Y$ if the value of $X$ is much smaller that that of $Y$, $X \gg Y$ if the value of $X$ is much larger than that of $Y$, and $X \approx Y $ if  $X \ll Y$ and $Y \ll X$ are not true. We write
$X = o(Y)$ if there exists a constant $0 < c \ll 1$ such that  $|X| \leq c |Y|$. We define
$X+ = X + \epsilon$ for $ 0 < \epsilon \ll 1$.\\
We also write $X\lesssim Y$ if $X\le CY$ for some $C>0$ that does not depend on
$X$ and $Y$. 
If $Y+$ appears in the mathematical expression
$X \leq C Y+$ then $ C $ may depend on $ \epsilon $ but not on $X$ and $Y$.
The origin of $\m R^N$ is denoted by
$O$ (this is mostly used to
denote by $B(O,r)$, the ball centered at the origin with radius $r>0$).
Unless otherwise specified, we let in the sequel $f$ (resp. $u$) be a function depending on the space variable (resp.
the space variable and the time variable). Unless otherwise specified, for sake of simplicity, we do not mention the spaces to which $f$ and $u$
belong in the estimates: this exercise is left to the reader.

\section{Proof of the lower bound in (\ref{Eqn:BlowUpRate})}
\label{Sec:LowerBoundProof}

Let $u$ be a solution of \eqref{equ}-\eqref{pert2}
with data $ \left( u(0),\partial_{t}u(0) \right) := (u_{0},u_{1}) \in H^{1}_{loc} \times L^{2}_{loc}(\mathbb{R}^{N})$ and with a $1-$ Lipschitz  graph. \\
\\
We claim that the ``$\leq $'' inequality of (\ref{Eqn:BlowUpRate})
holds. We may assume
without loss of generality
that $k$ is small enough so that all the statements below are true.
Let $ A (x):= \log^{- \frac{a}{p-1}} (- \log(x))$ if $x \ll  1 $ and  $A(x) := 1 $ if $x \gtrsim 1 $. \\
Assume that this is not the case. This means that there exists  $t_{1} \in [0, T(x_{0}))$ such that

\begin{equation}
\begin{array}{l}
\frac{\| u(t_{1}) \|_{L^{2} \left( B(x_{0}, T(x_{0}) - t_{1}) \right) } }{ \left(  T(x_{0}) - t_{1} \right)^{\frac{N}{2} - \frac{2}{p-1}}} +
\frac{ \| \nabla u(t_{1}) \|_{L^{2} \left( B(x_{0}, T(x_{0}) - t_{1}) \right)} +  \| \partial_{t} u(t_{1}) \|_{L^{2} \left( B(x_{0}, T(x_{0}) - t_{1}) \right)} } {\left(  T(x_{0}) - t_{1} \right)^{\frac{N}{2}- \frac{p+3}{p-1}}} \leq k A \left( T(x_{0})- t_{1} \right) \cdot
\end{array}
\nonumber
\end{equation}
Since $x_{0}$ is a non-characteristic point, we can use the dominated convergence theorem to infer that there exists $ 1 \gg \epsilon > 0$  such that

\begin{equation}
\begin{array}{l}
\frac{\| u(t_{1}) \|_{L^{2} \left( B \left( x_{0}, (T(x_{0}) - t_{1}) (1 + \epsilon) \right) \right) } }{ \left(  T(x_{0}) - t_{1} \right)^{\frac{N}{2} - \frac{2}{p-1}}} +
\frac{ \| \nabla u(t_{1}) \|_{L^{2} \left( B(x_{0}, T(x_{0}) - t_{1}) (1 + \epsilon) \right)} +
\| \partial_{t} u(t_{1}) \|_{L^{2} \left( B(x_{0}, T(x_{0}) - t_{1}) (1+ \epsilon) \right)} } {\left(  T(x_{0}) - t_{1} \right)^{\frac{N}{2}- \frac{p+3}{p-1}}} \leq 2 k A \left( T(x_{0})- t_{1} \right) \cdot
\end{array}
\nonumber
\end{equation}
A computation (using an appropriate change of variable) shows that for $\lambda = T(x_{0}) - t_{1}$, we get
\begin{equation}
\begin{array}{l}
\left\| \left(  f_{\lambda} ,  g_{\lambda} \right) \right\|_{H^{1}  \left( B(O,1 + \epsilon) \right) \times L^{2}  \left( B(O,1 + \epsilon) \right)} \lesssim k A (\lambda)
\end{array}
\label{Eqn:Estulambdalambdat1}
\end{equation}
with $ \left( f_{\lambda}(x), g_{\lambda}(x) \right) := \lambda^{\frac{2}{p-1}} \left( u \left(  t_{1}, \lambda x + x_{0}  \right),
 \lambda \partial_{t} u  \left( t_{1} , \lambda x + x_{0} \right) \right) $.
Let  $C^{'} \gg 1$ be a constant large enough such that all the statements below are true. Let $K:= \left\{ (t,x): t \in [0,1 + \epsilon), \, \text{and} \, t < 1 + \epsilon - |x| \right\}$.  We claim that there exists a unique solution \\
 $ v \in \mathcal{X} := \left\{ w: \; w \in \mathcal{C}_{t} H^{1} (K) \cap \mathcal{C}^{1}_{t} L_{x}^{2} (K), \; \text{and} \;
\| v \|_{\mathcal{X}} \leq C^{'} k A(\lambda)  \right\} $  such that  $ v = \tilde{\Psi}(v) $  with
$ \| v \|_{\mathcal{X}} := \sup \limits_{t \in [ 0,1 + \epsilon) } \sup \limits_{ 1 + \epsilon - t > R > 0}
\left( \| v(t) \|_{H^{1}(B(O,R))},  \| \partial_{t} v(t) \|_{L^{2}(B(O,R))} \right)  $ and

\begin{equation}
\begin{array}{l}
\tilde{\Psi}(w)(t) :=  \partial_{t} \mathcal{R}(t)* f_{\lambda} + \mathcal{R}(t) * g_{\lambda} + \int_{0}^{t} \mathcal{R}(t-s)*
h_{\lambda} \left( w(s) \right) \; ds \cdot
\end{array}
\nonumber
\end{equation}
Here $ h_{\lambda}(f) :=  |f|^{p-1} f \log^{a} \left( \log \left( 10 + \lambda^{-\frac{4}{p-1}} f^{2} \right) \right)  $.

Let $\left( w, w_{1}, w_{2} \right) \in \mathcal{X}^{3}$. In the sequel, we implicitly use elementary estimates involving the logarithm \footnote{such as $\log(xy) = \log(x) + \log(y)$;
$\log(x^{-1})= - \log(x) $;   $\log(x) \lesssim x^{0+}$, $\log(x) \lesssim x^{0-}$  for $x \gtrsim 1$, $x \lesssim 1$ respectively .} and the Sobolev embedding
$ H^{1} \left( B (O,R) \right)  \hookrightarrow L^{s} \left( B (O,R) \right)$ for $ 2 \leq s \leq (2p)+$. \\
We first estimate $ \| \tilde{\Psi}(w_{2}) - \tilde{\Psi}(w_{1}) \|_{\mathcal{X}} $. Assume that $\lambda \gtrsim 1$. We have
$ \| \tilde{\Psi}(w_{2}) - \tilde{\Psi}(w_{1}) \|_{\mathcal{X}} \lesssim   \left\| h_{\lambda}(w_{1}) - h_{\lambda}(w_{2}) \right\|_{L_{t}^{1} L_{x}^{2}(K)} $. Hence, the fundamental theorem of calculus and the H\"older inequality show that
\begin{equation}
\begin{array}{ll}
\| \tilde{\Psi}(w_{2}) - \tilde{\Psi}(w_{1}) \|_{X} & \lesssim  (1 + \epsilon) \sum \limits_{\substack{\bar{w} \in \{ w_{1}, w_{2} \} \\ r \in \{  p-1, (p-1)+ \} }}
\| \bar{w} \|^{r}_{L_{t}^{\infty} L_{x}^{2r} (K)} \| w_{2} - w_{1} \|_{\mathcal{X}} \cdot
\end{array}
\nonumber
\end{equation}
Assume that $\lambda \ll 1$. Let $a > 0$. Then

\begin{equation}
\begin{array}{l}
\| \tilde{\Psi}(w_{2}) - \tilde{\Psi}(w_{1}) \|_{\mathcal{X}} \\
\lesssim (1 + \epsilon)
\left(
\sum \limits_{\substack{\bar{w} \in \{ w_{1}, w_{2} \} \\ r \in \{ p \pm , (p-1)\pm \}}}
\| \bar{w} \|^{r}_{L_{t}^{\infty} L_{x}^{2r}(K)} + \log^{a} \left(- \log ( \lambda) \right)
\sum \limits_{\bar{w} \in \{ w_{1},w_{2} \}}  \| \bar{w} \|^{p-1}_{L_{t}^{\infty} L_{x}^{2(p-1)}(K)}
\right) \| w_{2} - w_{1} \|_{\mathcal{X}} \cdot
\end{array}
\nonumber
\end{equation}
Let $a < 0 $. Let $\bar{w} \in \{ w_{1}, w_{2} \}$. Then $ \left\| h_{\lambda}(w_{1}) - h_{\lambda}(w_{2}) \right\|_{L_{t}^{1} L_{x}^{2}(K)} $ is bounded by terms of a finite sum of terms of the form $ X_{1}(\lambda) $ and $ X_{2}(\lambda) $ with

$ X_{1}(\lambda) := \left\| \log^{a} \left( \log \left( 10 + \frac{\bar{w}^{2}}{\lambda^{\frac{4}{p-1}}} \right) \right) |\bar{w}|^{p-1} \left( w_{2} - w_{1} \right)
\right\|_{L_{t}^{1} L_{x}^{2} ( K \cap \{ \bar{w}^{2} \geq  \lambda^{\frac{4}{p-1}-} \} )} $, and \\
$ X_{2}(\lambda) := \left\|  \log^{a} \left( \log \left( 10 + \frac{\bar{w}^{2}}{\lambda^{\frac{4}{p-1}}} \right) \right) |\bar{w}|^{p-1} \left( w_{2} - w_{1} \right) \right\|_{L_{t}^{1} L_{x}^{2} ( K  \cap \{ \bar{w}^{2} \leq \lambda^{\frac{4}{p-1}-}\}
)}$. We have

\begin{equation}
\begin{array}{ll}
X_{1}(\lambda)  &  \lesssim (1 + \epsilon) \log^{a} \left( -\log(\lambda) \right)
\| \bar{w} \|^{p-1}_{L_{t}^{\infty} L_{x}^{2(p-1)} (K) }
\| w_{2} - w_{1} \|_{\mathcal{X}} \cdot
\end{array}
\nonumber
\end{equation}
We have

\begin{equation}
\begin{array}{ll}
X_{2} (\lambda) & \lesssim (1 + \epsilon) \| \bar{w} \|^{p-1}_{L_{t}^{\infty} L_{x}^{\infty} ( K  \cap \{ \bar{w}^{2} \leq \lambda^{\frac{4}{p-1}-}  \} )}
\| w_{2} - w_{1} \|_{L_{t}^{\infty} L_{x}^{2}(K)} \\
& \lesssim (1 + \epsilon) \lambda^{0+} \| w_{2} - w_{1} \|_{\mathcal{X}}.
\end{array}
\nonumber
\end{equation}
Hence, the estimates above show that  $\tilde{\Psi}$ is a contraction. \\
\\
We then estimate $\left\| \tilde{\Psi}(w)  \right\|_{\mathcal{X}}$.  Assume that $\lambda \gtrsim 1$. Then  (\ref{Eqn:LocNrjEst}), (\ref{Eqn:Estulambdalambdat1}), and the H\"older inequality show that

\begin{equation}
\begin{array}{ll}
\left\| \tilde{\Psi}(w)  \right\|_{\mathcal{X}}
& \lesssim  k A(\lambda) + (1 + \epsilon) A(\lambda) \sum \limits_{ r \in \{ p ,p+ \} }
\| w \|^{r}_{L_{t}^{\infty} H^{1}(K)}.
\end{array}
\label{Eqn:EstTildePsi}
\end{equation}
Assume that $\lambda \ll 1$. Let $a > 0$. Then (\ref{Eqn:EstTildePsi}) holds. Let $a < 0$. Let  $X(w) :=\left\| |v|^{p-1} v \log^{a} \left(  \log \left( 10 +  \lambda^{-\frac{4}{p-1}} w^{2} \right) \right) \right\|_{L_{t}^{1} L_{x}^{2}(K)}$. Then $ X(w) \lesssim  X_{1}(w) + X_{2}(w) $  with

\begin{equation}
\begin{array}{l}
X_{1}(w) := \left\| |w|^{p-1} w \log^{a} \left(  \log \left( 10 +  \lambda^{-\frac{4}{p-1}} w^{2} \right) \right) \right\|_{L_{t}^{1} L_{x}^{2} (K \cap
\{ w^{2} \geq \lambda^{\frac{4}{p-1}-}  \} )}, \; \text{and} \\
X_{2}(w) := \left\| |w|^{p-1} w \log^{a}
\left(  \log \left( 10 +  \lambda^{- \frac{4}{p-1}} w^{2} \right) \right) \right\|_{L_{t}^{1} L_{x}^{2}(K
\cap \{ w^{2} \leq \lambda^{\frac{4}{p-1}-} \} )} \cdot
\end{array}
\nonumber
\end{equation}
We have $ X_{1}(w) \lesssim (1+ \epsilon) \log^{a} \left(  - \log (\lambda) \right)  \| w \|^{p}_{L_{t}^{\infty} H^{1}(K)} $. By interpolation there exists $\theta:= \theta(p) > 0 $ such that

\begin{equation}
\begin{array}{ll}
X_{2}(w) & \lesssim (1+ \epsilon) \| w \|^{p}_{L_{t}^{\infty} L_{x}^{2p} (K  \cap \{
w^{2}  \leq \lambda^{\frac{4}{p-1}-}  \}) } \\
& \lesssim (1+ \epsilon)
\| w \|^{p \theta}_{L_{t}^{\infty} L_{x}^{\infty}( K \cap \{ w^{2} \leq \lambda^{\frac{4}{p-1}-}  \} )}
\| w \|^{p (1- \theta)}_{L_{t}^{\infty} L_{x}^{2} (K) } \\
& \lesssim (1+ \epsilon) \lambda^{0+}  \| w \|^{p(1- \theta)}_{L_{t}^{\infty} H^{1} (K)}
\end{array}
\nonumber
\end{equation}
Hence from the estimates above  we see that $\tilde{\Psi}(\mathcal{X}) \subset \mathcal{X} $. \\
\\
Hence by the fixed point theorem we see that there exists $v \in \mathcal{X}$ such that $ v = \tilde{\Psi}(v) $.\\
\\
Let \\
$\Omega := \left\{ (t,x): \, 0 \leq  t  <
T(x_{0}) + \epsilon \left( T(x_{0}) - t_{1} \right),  \;
|x - x_{0}| < T(x_{0}) + \epsilon \left( T(x_{0}) -t_{1} \right) - t \right\}$. Let $w$ be defined by
$ w(t) := u(t)$ if $ t \in [0, t_{1}]$ and $ w (t) := \frac{1}{\lambda^{\frac{2}{p-1}}} v \left( \frac{t- t_{1}}{\lambda}, \frac{x-x_{0}}{\lambda} \right) $ if $t \in [ t_{1}, T(x_{0})  + \epsilon (T(x_{0}) - t_{1})  ) $. By the
finite speed of propagation we infer that $ w(t) = u(t)$ on $ \Omega \cap ( t_{1} - \alpha, t_{1} + \alpha) $ for $\alpha > 0 $ small.  Hence $ w \in \mathcal{C}_{t} H^{1} (\Omega) \cap \mathcal{C}_{t}^{1} L_{x}^{2} (\Omega)$ and it satisfies $\Psi(w)=w$, with
$\Psi$ defined in (\ref{Eqn:DefPsi}) . This contradicts the definition of $T(x_{0})$.

\section{Design of a Lyapunov functional in similarity variables}
\label{Sec:DesLyap}
This section is at the heart of our argument. We design here a
Lyapunov functional for the similarity variables' PDE
\eqref{Eqn:Wavew}, based on the functional $\q E(w(s))$ \eqref{Eqn:Energyw}, and prove its nonnegativity, thanks to a blow-up
criterion. This way, the Lyapunov functional will turn to be bounded
from above and below. Using some intricate energy and nonlinear arguments, we prove the boundedness of the solution in $H^1\times L^2(B(O,1))$ in
time averages, which yields Proposition
\ref{Prop:BoundPolynws}. If we follow the spirit of earlier papers on
the subject, the $\log \log$ factor in the nonlinearity \eqref{pert2}
makes our argument very delicate.

\medskip

We proceed in 4 subsections, first bounding the derivative of $\q
E(w(s))$ \eqref{Eqn:Energyw}, then introducing a corrective
term. Combining it with $\q E(w(s))$, we define a Lyapunov functional
in the third subsection and prove its nonnegativity. Finally, the last subsection is devoted to
the proof of  Proposition \ref{Prop:BoundPolynws}. \\

\subsection{Estimate of derivative of the energy}

The proposition below yields an estimate of the derivative of $\mathcal{E}(w)$:
\begin{lem}[Upper bound on the derivative of  $\q E(w(s))$
  \eqref{Eqn:Energyw}]
  Consider $w$ a solution of equation \eqref{Eqn:Wavew}.
Let $s \gg 1$. Then there exists a large positive constant $C$ such that

\begin{equation}
\begin{array}{l}
\frac{d \mathcal{E}(w)}{ds} \leq - \frac{3 \alpha}{2} \int_{B(O,1)} (\partial_{s} w)^{2} \frac{1}{1 - |y|^{2}} \; \rho dy +
\frac{C}{s \log^{a+1}(s)} \int_{B(O,1)} |w|^{p+1} g(\phi w) \; \rho
  dy + \Sigma,
\end{array}
\label{Eqn:EstDerivE}
\end{equation}
where $\phi$ is defined in \eqref{Eqn:Defphi} and
\begin{equation}
\begin{array}{l}
\Sigma \leq \frac{C}{s^{2}}  \int_{B(O,1)} |\nabla w|^{2} (1 - |y|^{2}) \; \rho  dy
+ \frac{C}{s^{2}} \int_{B(O,1)} w^{2} \; \rho  dy + C e^{-s} \cdot
\end{array}
\nonumber
\end{equation}

\label{Lem:EstDerivE}
\end{lem}
\begin{proof}

We multiply (\ref{Eqn:Wavew})  by  $ \partial_{s} w  \rho $ and we integrate over $B(O,1)$. We get

\begin{equation}
\begin{array}{l}
\Sigma_{l} = \Sigma_{r,1} + ... + \Sigma_{r,8} \nonumber
\end{array}
\nonumber
\end{equation}
Here  $ \Sigma_{l} :=  \int_{B(O,1)} \partial_{s}^{2} w \rho \partial_{s} w \; dy $, $ \Sigma_{r,2} := \int_{B(O,1)}
\frac{1}{\rho} \dive \left( \rho \nabla w - \rho (y \cdot \nabla w) y \right) \rho \partial_{s} w \; dy $, \\
$ \Sigma_{r,3}  := \frac{2a}{(p-1)s \log(s)} \int_{B(O,1)}  (y \cdot \nabla w) \rho \partial_{s} w \; dy $,
$ \Sigma_{r,4} := - \frac{2(p+1)}{(p-1)^{2}} \int_{B(O,1)} w \partial_{s} w  \rho \; dy $, \\
$ \Sigma_{r,5} := \gamma(s) \int_{B(O,1)} w \partial_{s} w \; \rho \; dy $,  $\Sigma_{r,6} := - \left( \frac{p+3}{p-1} - \frac{2a}{(p-1)s \log(s)} \right) \int_{B(O,1)} (\partial_{s} w)^{2} \rho \; dy $, \\
$\Sigma_{r,7} : = - 2 \int_{B(O,1)} y \cdot \nabla \partial_{s} w \rho \; dy $ and\\  $\Sigma_{r,8} := - e^{-\frac{2sp}{p-1}} \log^{\frac{a}{p-1}}(s) \int_{B(O,1)} f(\phi w)  \partial_{s} w  \rho \; dy  $. \\
Integrating by parts whenever it is necessary, and using the inequality $ab \leq \frac{a^{2}}{2} + \frac{b^{2}}{2}$, we get
\begin{equation}
\begin{array}{l}
\Sigma_{r,1} := \frac{d}{ds} \left(  \frac{1}{2} \int_{B(O,1)}
  (\partial_{s} w)^{2} \rho \; dy \right), \\
  \Sigma_{r,2} = - \frac{d}{ds} \left( \frac{1}{2} \int_{B(O,1)} \left( |\nabla w|^{2}  - |\nabla w \cdot y|^{2} \right) \rho \; dy \right), \\
\\
\left| \Sigma_{r,3} \right| \lesssim  \frac{1}{s^{2}} \int_{B(O,1)} |\nabla w|^{2} (1 - |y|^{2}) \rho \; dy + \frac{1}{\log^{2}(s)} \int_{B(O,1)} (\partial_{s} w )^{2} \; \frac{\rho}{1-|y|^{2}} \; dy, \\
\\
\Sigma_{r,4}  = - \frac{d}{ds} \left(  \int_{B(O,1)} \frac{p+1}{(p-1)^{2}} w^{2} \rho \; dy \right), \\
\\
\left| \Sigma_{r,5}(s) \right| \lesssim \frac{1}{s^{2}} \int_{B(O,1)} w^{2} \rho \; dy
+ \frac{1}{\log^{2}(s)} \int_{B(O,1)} (\partial_{s} w )^{2} \; \frac{\rho}{1-|y|^{2}} \; dy, \; \text{and} \\
\\
\Sigma_{r,6}(s) = - \frac{p+3}{p-1} \int_{B(O,1)} (\partial_{s} w)^{2} \rho \; dy +
O \left( \frac{1}{s \log(s)} \right) \int_{B(O,1)} (\partial_{s} w)^{2} \rho \; dy  \cdot
\end{array}
\nonumber
\end{equation}
We then estimate $ \Sigma_{r,7}(s) $. Observe from (\ref{Eqn:ValAlpha}) that  $ 2 \alpha + N = \frac{p+3}{p-1} $. Hence

\begin{equation}
\begin{array}{ll}
\Sigma_{r,7} & = - \int_{B(O,1)} y \cdot \nabla \left( (\partial_{s} w)^{2} \right) \rho \; dy \\
& = \frac{p+3}{p-1} \int_{B(O,1)} (\partial_{s} w)^{2} \rho \; dy - 2 \alpha  \int_{B(O,1)} \frac{(\partial_{s} w)^{2}}{1- |y|^{2}} \rho \; dy
\end{array}
\nonumber
\end{equation}
We then estimate $\Sigma_{r,8} $. We have $\Sigma_{r,8} = I + J $ with

\begin{equation}
\begin{array}{l}
I := - \int_{B(O,1)} e^{- \frac{2sp}{p-1}} \log^{\frac{a}{p-1}}(s) \phi^{-1}  \frac{d}{ds} \left( F(\phi w) \right) \rho \; dy, \; \text{and} \\
J :=  \int_{B(O,1)} e^{- \frac{2sp}{p-1}} \log^{\frac{a}{p-1}}(s)  F^{'}(\phi w) \phi^{-1} \phi^{'} w  \rho \; dy \cdot
\end{array}
\nonumber
\end{equation}
By integrating by parts, we see that
\begin{equation}
\begin{array}{ll}
I & =  - \frac{d}{ds} \left( \int_{B(O,1)} e^{-\frac{2(p+1)s}{p-1}} \log^{\frac{2a}{p-1}}(s) F(\phi w) \rho \; dy  \right)\\
&- \frac{2(p+1)}{p-1} e^{- \frac{2(p+1)s}{p-1}}  \log^{\frac{2a}{p-1}}(s) \int_{B(O,1)} F(\phi w)  \rho \; dy \\
& + \frac{2a}{p-1}  e^{- \frac{2(p+1)s}{p-1}} \frac{ \log^{\frac{2a}{p-1} - 1}(s)}{s}  \int_{B(O,1)} F(\phi w) \rho \; dy \cdot
\end{array}
\nonumber
\end{equation}
We also see from (\ref{Eqn:Defphi}) that

\begin{equation}
\begin{array}{rl}
J =& \frac{2(p+1)}{p-1} e^{- \frac{2(p+1)s}{p-1}} \log^{\frac{2a}{p-1}}(s) \int_{B(O,1)} \frac{\phi w f(\phi w)}{p+1} \rho \; dy\\
&- \frac{2a}{(p-1)s} e^{- \frac{2(p+1)s}{p-1}} \log^{\frac{2a}{p-1}-1}(s)  \int_{B(O,1)} \frac{\phi w f(\phi w)}{2} \rho \; dy \cdot
\end{array}
\nonumber
\end{equation}
Hence,
\begin{equation}
\begin{array}{l}
- \Sigma_{r,8} = \frac{d}{ds} \left( \int_{B(O,1)} e^{-\frac{2(p+1)s}{p-1}} \log^{\frac{2a}{p-1}}(s) F(\phi w) \rho \; dy  \right) + X + Y, \; \text{with}
\end{array}
\nonumber
\end{equation}

\begin{equation}
\begin{array}{l}
X :=  \frac{2(p+1)}{p-1} e^{- \frac{2(p+1) s}{p-1}} \log^{\frac{2a}{p-1}}(s) \int_{B(O,1)} \left( F(\phi w) - \frac{\phi w f(\phi w) }{p+1} \right) \rho \; dy \; \text{and} \\
Y := - \frac{2a}{p-1} e^{- \frac{2(p+1)s}{p-1}} \frac{\log^{\frac{2a}{p-1} - 1}(s)}{s} \int_{B(O,1)}  \left( F(\phi w) - \frac{\phi w f(\phi w) }{2} \right) \rho \; dy \cdot
\end{array}
\label{Eqn:DefXY}
\end{equation}
It remains to estimate $X$ and $Y$.  To this end we estimate $ F(\phi w) - \frac{\phi w  f(\phi w)}{p+1} $  by using
(\ref{Eqn:Defphi}), (\ref{Eqn:DefF1}), (\ref{Eqn:DecompF}), (\ref{Eqn:ApproxPot}), $|F(x)| \lesssim |x|^{p+1} \left( 1 + g(x) \right)$, and elementary estimates. Since
$ \{ \phi w^{2} \gg 1 \}  \subset \{ \phi^{2} w^{2} \gtrsim \phi \} $ and  $ \{ \phi w^{2}  \lesssim 1 \}  \subset \{ \phi w  \lesssim \phi^{\frac{1}{2}} \} $ we get

\begin{equation}
\begin{array}{l}
\phi w^{2} \gg 1: \;  \left| F(\phi w) - \frac{\phi w f(\phi w)}{p+1} \right| \lesssim \frac{  \left| \phi w  f(\phi w ) \right|}{ s \log(s) }  \\
\phi w^{2} \lesssim 1: \; \left| F(\phi w) - \frac{\phi w f(\phi w)}{p+1} \right| \lesssim \phi^{\frac{p+1}{2}} \left(  1 + g(x) \right) \cdot
\end{array}
\label{Eqn:ErrorFPhi} 
\end{equation}
Hence

\begin{equation}
\begin{array}{l}
X \lesssim \frac{1}{s \log^{a+1}(s)}  \int_{B(O,1)} |w|^{p+1} g(\phi w)  \rho \; dy + e^{-s}  \cdot
\end{array}
\nonumber
\end{equation}
We also see from (\ref{Eqn:ApproxPot}) and the triangle inequality that

\begin{equation}
\begin{array}{l}
Y \lesssim \frac{1}{s \log^{a+1}(s)}  \int_{B(O,1)} |w|^{p+1} g(\phi w)  \rho \; dy  \cdot
\end{array}
\nonumber
\end{equation}
Hence, collecting all the estimates above we see that
(\ref{Eqn:EstDerivE}) holds. This concludes the proof of Lemma \ref{Lem:EstDerivE}.




\end{proof}

\subsection{Corrective term}
In this subsection, we introduce a corrective term which will be combined to the
functional  $\q E(w(s))$ \eqref{Eqn:Energyw} in the next subsection to
provide us with a Lyapunov functional for equation
\eqref{Eqn:Wavew}. More precisely, introducing
\begin{equation}
\begin{array}{l}
\mathcal{J}(w) : s \rightarrow \mathcal{J}(w(s)) := - \frac{1}{s \log(s)} \int_{B(O,1)} w(s) \partial_{s} w(s)  \rho \; dy,
\end{array}
\label{Eqn:DefJ}
\end{equation}
we have the following statement:
\begin{lem}[Upper bound on the derivative  of the corrective term]
Let $s \gg 1$. There exists a positive constant $C > 0$ such that

\begin{equation}
\begin{array}{ll}
\frac{d \mathcal{J}(w)}{ds} & \leq \frac{p+3}{2 s \log(s)} \mathcal{E}(w)  - \frac{p+7}{4 s \log(s)} \int_{B(O,1)} (\partial_{s} w)^{2} \rho \; dy \\
& - \frac{p+1}{2(p-1) s \log(s)} \int_{B(O,1)} w^{2} \rho \; dy  - \frac{p-1}{4s \log(s)} \int_{B(O,1)}  \left( |\nabla w|^{2} - (y \cdot \nabla w)^{2} \right) \rho \; dy \\
& - \frac{p-1} {2(p+1) s \log^{a+1}(s)} \int_{B(O,1)} |w|^{p+1} g(\phi w) \rho \; dy  + \Sigma, \; \text{with}
\end{array}
\label{Eqn:DerivJ}
\end{equation}

\begin{equation}
\begin{array}{l}
\Sigma \leq \frac{C}{s^{\frac{1}{2}}} \int_{B(O,1)} \frac{(\partial_{s} w)^{2}}{1-|y|^{2}} \rho \; dy
+ \frac{C}{s^{\frac{3}{2}} \log^{2}(s)} \int_{B(O,1)} w^{2} \rho \; dy \\
+ C \frac{1}{s^{2} (\log(s))^{a+2}} \int_{B(O,1)} |w|^{p+1} g(\phi w) \rho \; dy + C  e^{-s},
\end{array}
\nonumber
\end{equation}
where $\phi$ is introduced in \eqref{Eqn:Defphi}.
\label{Lem:EstDerivJ}
\end{lem}

\begin{proof}

We have  $ \frac{d \mathcal{J}(w)}{ds}=  - \frac{1}{s \log(s)} \int_{B(O,1)} (\partial_{s} w)^{2} \rho \; dy
+ \Sigma_{1} + \Sigma_{2} $ with

\begin{equation}
\begin{array}{l}
\Sigma_{1} := - \frac{1}{s \log(s) }  \int_{B(O,1)} w \partial_{s}^{2} w \rho \; dy, \; \text{and} \\
\Sigma_{2} := \frac{1}{s^{2} \log(s)} \left( 1 + \frac{1}{\log(s)} \right) \int_{B(O,1)} w \partial_{s} w \rho \; dy \cdot
\end{array}
\nonumber
\end{equation}
We estimate $\Sigma_{2}$. The inequality $ab \leq \frac{a^{2}}{2} + \frac{b^{2}}{2}$ shows that

\begin{equation}
\begin{array}{l}
\Sigma_{2} \leq \frac{1}{2s^{2} \log(s)}  \int_{B(O,1)} w^{2} \rho \; dy + \frac{1}{2s^{2} \log(s)} \int_{B(0,1)} (\partial_{s} w)^{2} \rho \; dy \cdot
\end{array}
\nonumber
\end{equation}
We now estimate $\Sigma_{1}$. We see from (\ref{Eqn:Wavew}) that

\begin{equation}
\Sigma_{1} = I_{1}+...+ I_{6} ,
\end{equation}
with $I_{1} := -\frac{1}{s \log(s)} \int_{B(O,1)} \frac{1}{\rho} \dive \left( \rho \nabla w  - \rho (y \cdot \nabla w ) y \right) w \rho \; dy  $, \\
$ I_{2} := - \frac{2a}{(p-1)s^{2} \log^{2}(s)} \int_{B(O,1)}  y \cdot \nabla w w  \rho \; dy $,
$ I_{3} := \frac{1}{s \log(s)} \left( \frac{2(p+1)}{(p-1)^{2}}  - \gamma \right) \int_{B(O,1)}  w^{2}   \rho \; dy  $, \\
$ I_{4} :=  \frac{1}{s \log(s)} \left( \frac{p+3}{p-1} - \frac{2a}{(p-1)s \log(s)} \right) \int_{B(O,1)}  \partial_{s} w  w \rho \; dy $, \\
$ I_{5}  := \frac{2}{s \log(s)} \int_{B(O,1)} y \cdot \nabla \partial_{s} w  w \rho \; dy  $,  and\\
$ I_{6} := - \frac{1}{s \log(s)} e^{-\frac{2ps}{p-1}} \log^{\frac{a}{p-1}}(s) \int_{B(O,1)} f(\phi w) w \rho \; dy $. \\
An integration by parts shows that

\begin{equation}
\begin{array}{l}
I_{1} = \frac{1}{s \log(s)} \int_{B(O,1)} \left( |\nabla w|^{2} - |y \cdot \nabla w|^{2} \right) \rho \; dy \cdot
\end{array}
\nonumber
\end{equation}
Writing  $w\nabla w = \nabla \left( \frac{w^{2}}{2}\right) $ we have

\begin{equation}
\begin{array}{l}
I_{2} = \frac{a}{(p-1)s^{2} \log^{2}(s)} \left(  N \int_{B(O,1)} w^{2} \rho \; dy - \int_{B(O,1)} \frac{2 |y|^{2} \alpha}{1- |y|^{2}} w^{2} \rho \; dy \right)
\end{array}
\nonumber
\end{equation}
Hence $  | I_{2} | \lesssim  \frac{1}{s^{2}\log^{2}(s)} \int_{B(O,1)}
w^{2} \rho \; dy $. \\
Clearly, $|I_{3}| \leq \left( \frac{2(p+1)}{(p-1)^{2} s\log(s)} + O \left( \frac{1}{s^{2} \log^{2}(s)} \right) \right) \int_{B(O,1)} w^{2} \rho \; dy $. The inequality
$AB \leq \frac{A^{2}}{2} + \frac{B^{2}}{2} $
shows that

\begin{equation}
\begin{array}{l}
I_{4} \lesssim \frac{1}{s^{\frac{1}{2}}} \int_{B(O,1)} (\partial_{s} w)^{2} \rho \; dy
+ \frac{1}{s^{\frac{3}{2}} \log^{2}(s)} \int_{B(O,1)} w^{2} \rho \; dy \cdot
\end{array}
\nonumber
\end{equation}
From $\nabla \partial_{s} w  w = \frac{1}{2} \nabla \left( \partial_{s} (w^{2}) \right) $  we have

\begin{equation}
\begin{array}{ll}
I_{5} & = - \frac{2N}{s \log(s)} \int_{B(O,1)} w \partial_{s} w \rho \; dy  + \frac{2 \alpha}{s \log(s)} \int_{B(O,1)} w \partial_{s} w \frac{|y|^{2}}{1- |y|^{2}}
\rho \; dy \\
& = X + Y
\end{array}
\nonumber
\end{equation}
We have

\begin{equation}
\begin{array}{l}
X \lesssim \frac{1}{s^{\frac{3}{2}} \log^{2}(s)} \int_{B(O,1)} w^{2} \rho \; dy + \frac{1}{s^{\frac{1}{2}}} \int_{B(O,1)} (\partial_{s} w)^{2} \rho \; dy \cdot
\end{array}
\nonumber
\end{equation}
In order to estimate $Y$ we use the Hardy-type inequality  (see e.g \cite{MZajm03})
\begin{equation}
\begin{array}{l}
\int_{B(O,1)} w^{2} \frac{|y|^{2}}{1- |y|^{2}} \rho \; dy \lesssim \int_{B(O,1)} |\nabla w|^{2} (1-|y|^{2}) \rho \; dy
+ \int_{B(O,1)} w^{2} \rho \; dy
\end{array}
\label{Eqn:HardyType}
\end{equation}
to get

\begin{equation}
\begin{array}{l}
Y \lesssim \frac{1}{s^{\frac{3}{2}} \log^{2}(s)} \int_{B(O,1)} w^{2} \rho \; dy + \frac{1}{s^{\frac{1}{2}}} \int_{B(O,1)}
\frac{(\partial_{s} w)^{2}}{1- |y|^{2}} \rho \; dy \cdot
\end{array}
\nonumber
\end{equation}
We turn now our attention to $I_{6}$. We get from (\ref{Eqn:Defphi})

\begin{equation}
\begin{array}{ll}
I_{6} & =  - \frac{1}{s \log^{a+1}(s)} \int_{B(O,1)} |w|^{p+1} g(\phi w) \rho \; dy \cdot
\end{array}
\nonumber
\end{equation}
Hence we see from the estimates above and from (\ref{Eqn:ErrorFPhi})
that (\ref{Eqn:DerivJ}) holds. This concludes the proof of Lemma \ref{Lem:EstDerivJ}.

\end{proof}

\subsection{A new functional: decreasing and nonnegativity properties}
Using the previous subsections,
we introduce the following maps $\mathcal{N}_{m}(w)$
and $\mathcal{H}_{m}(w)$ (with $m \in \mathbb{R}^{+}$):
\begin{equation}
\begin{array}{ll}
\mathcal{N}_{m}(w): s \rightarrow \mathcal{N}_{m}(w(s)):=
  (\log(s))^{-\frac{m(p+3)}{2}} \mathcal{H}_{m}(w(s)) + m^{2} e^{-s},\\
\mathcal{H}_{m}(w) : s \rightarrow \mathcal{E}(w(s)) + m \mathcal{J}(w(s)) \cdot
\end{array}
\label{Eqn:DefHN}
\end{equation}
The lemma below shows that $ s \rightarrow \mathcal{N}_{m} (w(s))$ has a decreasing property (for a large $m$):
\begin{lem}[A nonnegative Lyapunov functional for equation
  \eqref{Eqn:Wavew}]\label{lemlyap}

There exist $m \gg 1$, $s_{0} \gg 1$, and $C > 0$ such that for all $s \geq s_{0} $ we have\\
\\
$(i)$:

\begin{equation}
\begin{array}{ll}
\mathcal{N}_{m}(w(s+1)) &- \mathcal{N}_{m}(w(s)) \leq - \frac{C}{s \log^{1+b}(s) } \int_{s}^{s+1} \int_{B(O,1)} \frac{(\partial_{\tau} w)^{2}}{1-|y|^{2}} \rho dy d \tau \\
& - \frac{C}{s \log^{1+b}(s)} \int_{s}^{s+1} \int_{B(O,1)} \left( |\nabla w|^{2} - |\nabla w \cdot y|^{2} \right) \rho  dy d \tau \\
& - \frac{C}{s \log^{a + 1 + b }(s)} \int_{s}^{s+1} \int_{B(O,1)} |w|^{p+1} \log^{a} \left( \log(10 + \phi^{2} w^{2}) \right) \rho  dy d \tau   \\
& - \frac{C}{\log^{1+b}(s)} \int_{s}^{s+1} \int_{B(O,1)} w^{2} \rho  dy, \; \text{where}
\end{array}
\label{Eqn:EstDecreaseAvN}
\end{equation}
$b := \frac{m(p+3)}{2}$. \\

$(ii)$:

\begin{equation}
\begin{array}{l}
\mathcal{N}_{m}(w(s)) \geq 0
\end{array}
\label{Eqn:EstPosN}
\end{equation}

\label{Lem:DecreaseN}
\end{lem}

\begin{proof}

(i) Let $s_{0} \gg 1$ be constant large enough such that all the
statements below are true. A computation shows that
\[
   \frac{d \mathcal{N}_{m}(w)}{ds} = (\log(s))^{-\frac{m(p+3)}{2}}
\left( \frac{d \mathcal{H}_{m}(w)}{ds} - \frac{m(p+3)}{2 s \log(s)}
  \mathcal{H}_{m}(w) \right) - m^{2} e^{-s}.
\]
Hence, by choosing $ m \gg 1$ large enough we
see from (\ref{Eqn:EstDerivE}), (\ref{Eqn:DerivJ}), and the estimate

\begin{equation}
\begin{array}{l}
\left| \mathcal{J}(w(s)) \right| \leq \frac{1}{2s \log(s)}
\left( \int_{B(O,1)} w^{2} \rho \; dy + \int_{B(O,1)} (\partial_{s} w)^{2} \rho \; dy \right)
\end{array}
\label{Eqn:EstJw}
\end{equation}
derived from $a b \leq \frac{a^{2}}{2} + \frac{b^{2}}{2}$ that there exists $ C > 0 $ such that

\begin{equation}
\begin{array}{ll}
\frac{d \mathcal{N}_{m}(w)}{ds} & \leq  - \frac{C}{\log^{b}(s)} \int_{B(0,1)} (\partial_{s} w)^{2} \; \rho  dy\\
&- \frac{C}{s \log^{1 + b }(s)} \int_{B(O,1)} \left( |\nabla w|^{2} - |\nabla w \cdot y|^{2} \right) \; \rho  dy \\
& - \frac{C}{s \log^{1+b}(s)} \int_{B(O,1)} w^{2} \rho \; dy - \frac{C}{s \log^{a+1 + b}(s)} \int_{B(O,1)} |w|^{p+1} g(\phi w) \; \rho  dy \cdot
\end{array}
\label{Eqn:EstDerivN}
\end{equation}
Hence integrating (\ref{Eqn:EstDerivN}) from $s$ to $s+1$ we get (\ref{Eqn:EstDecreaseAvN}). \\
(ii) In order to prove the second statement we argue by contradiction. Assume that there exists $S_{1} \geq s_{0} $ such that
$\mathcal{N}_{m} (w(S_{1})) < 0$. Let $ \delta > 0 $  be a constant so small that all the estimates below are true.
Let  $\tilde{w}^{\delta}(y,s) := w_{x, T^{*}(x) - \delta}(y,s)$. By (\ref{Eqn:SimilVar}) we have

\begin{equation}
\begin{array}{ll}
\tilde{w}^{\delta}(y,s) & := \frac{\phi \left(-\log(\delta + e^{-s}) \right) }{\phi(s)} w \left( \frac{y}{1 + \delta e^{s}}, -\log \left( \delta + e^{-s} \right) \right) \cdot
\end{array}
\label{Eqn:Reltildewandw}
\end{equation}
Here $ \phi(t) := \phi_{T^{*}(x) - \delta}(t)$.\\
Observe that $\tilde{w}^{\delta} $ is defined on $[S_{1} + 1, \infty)$, since
$- \log \left( \delta + e^{-( 1+S_{1} )} \right) \geq S_{1}$. Observe that

\begin{equation}
\begin{array}{l}
\mathcal{N}_{m} \left( \tilde{w}^{\delta}(1+ S_{1}) \right) < 0:
\end{array}
\label{Eqn:AssNNeg}
\end{equation}
this follows from the continuity of the map $ \delta^{'} \rightarrow \mathcal{N}_{m} \left( \tilde{w}^{\delta^{'}}(1+ S_{1}) \right)$ for $\delta^{'}$ positive and small enough. Since $ s \rightarrow N_{m} \left( \tilde{w}^{\delta}(s) \right) $ is decreasing for $s \geq S_{1} + 1$ we get

\begin{equation}
\begin{array}{l}
\lim \limits_{s \rightarrow \infty} \mathcal{N}_{m} \left( \tilde{w}^{\delta}(s) \right) \leq \mathcal{N}_{m} \left( \tilde{w}^{\delta}(1 + S_{1}) \right) \cdot
\end{array}
\nonumber
\end{equation}
We also see from the estimate $ab \leq \frac{a^{2}}{2} + \frac{b^{2}}{2}$ that

\begin{equation}
\begin{array}{l}
\int_{B(O,1)} \tilde{w}^{\delta} \partial_{s} \tilde{w}^{\delta} \; \rho  dy \lesssim  \int_{B(O,1)} ( \tilde{w}^{\delta})^{2} \; \rho  dy
+ \int_{B(O,1)} \left( \partial_{s} \tilde{w}^{\delta} \right)^{2} \; \rho dy \cdot
\end{array}
\label{Eqn:ContPerturbTildew}
\end{equation}
Hence

\begin{equation}
\begin{array}{ll}
\mathcal{N}_{m}  \left( \tilde{w}^{\delta}(s) \right) & \geq
- e^{-\frac{2(p+1)}{p-1} s} (\log(s))^{\frac{2a}{p-1}- b} \int_{B(O,1)} F(\phi \tilde{w}^{\delta}) \; \rho  dy
\end{array}
\nonumber
\end{equation}
We see from the boundedness of $F(\phi \tilde{w}^{\delta})$ on $ \{ \phi \tilde{w}^{\delta} \lesssim  1 \}$  and
$F(\phi \tilde{w}^{\delta}) \lesssim  | \phi \tilde{w}^{\delta}|^{(p+1)+} $ derived from (\ref{Eqn:ApproxPot}) on
$ \{ \phi \tilde{w}^{\delta} \gg 1 \} $, and (\ref{Eqn:Reltildewandw}), that there exists a large constant $C > 0$ such that

\begin{equation}
\begin{array}{ll}
\mathcal{N}_{m} \left( \tilde{w}^{\delta}(s) \right) & \geq - C \left(  1 + \delta e^{s} \right)^{N} e^{-\frac{2(p+1)s}{p-1}} (\log(s))^{\frac{2a}{p-1} - b}
\phi^{(p+1)+} \left( \log( \delta + e^{-s} ) \right) \times  \\
& \int_{B(O,1)} \left| w \left( z, -\log(\delta + e^{-s}) \right) \right|^{(p+1)+} \;  dy
 - C e^{-2s} \cdot
 \end{array}
\nonumber
\end{equation}
The estimate above and the embedding $ H^{1} \left( B(O,1) \right) \hookrightarrow  L^{(p+1)+}(B(O,1))  $ show that

\begin{equation}
\begin{array}{ll}
\lim \limits_{s \rightarrow \infty} \mathcal{N}_{m} \left( \tilde{w}^{\delta}(s) \right)  & = 0
\end{array}
\nonumber
\end{equation}
Given that the functional $s\mapsto \q N_m(\tilde w^\delta(s))$ is
decreasing by item (i), this is impossible, because of
(\ref{Eqn:AssNNeg}). Thus, a contradiction follows and item (ii)
follows. This concludes the proof of Lemma \ref{lemlyap}.
\end{proof}

\subsection{ Proposition \ref{Prop:BoundPolynws}: the proof}

In this subsection we prove Proposition \ref{Prop:BoundPolynws}. First we deduce from Lemma \ref{Lem:DecreaseN} a polynomial bound of averages (in space and in time) of $w$.

\begin{lem}
If  $s \geq s_{0}$  then the following holds:

\begin{equation}
\begin{array}{l}
0 \leq \mathcal{N}_{m} \left( w(s) \right) \leq \mathcal{N}_{m} \left( w(s_{0}) \right), \; \text{and}
\end{array}
\label{Eqn:EstDecN}
\end{equation}

\begin{equation}
\begin{array}{rl}
&\int_{s}^{s+1} \int_{B(O,1)} \left( |\nabla w|^{2} (1 - |y|^{2}) +  \frac{ \left( \partial_{\tau} w \right)^{2} }{1- |y|^{2}}
+ w^{2} \right)
  \rho \; dy \; d \tau\\
  \lesssim & \mathcal{N}_{m} \left( w(s_{0}) \right) \left( s \log^{1+b}(s) \right) \cdot
\end{array}
\label{Eqn:EstAverNablaW}
\end{equation}
Here $s_{0}$ is the number defined in Lemma \ref{Lem:DecreaseN}.

\label{Lem:BoundPoly}
\end{lem}

\begin{proof}
(\ref{Eqn:EstDecN}) follows from (\ref{Eqn:EstDerivN});
(\ref{Eqn:EstDecreaseAvN}) and (\ref{Eqn:EstDecN}) yield
(\ref{Eqn:EstAverNablaW}).

\end{proof}

We now prove Proposition \ref{Prop:BoundPolynws}. Recall from the
covering lemma (see e.g Proposition 3.4 page 1138 in \cite{MZimrn05}) that there exists
$C:= C(\delta_{0}) > 0 $ such that

\begin{equation}
\begin{array}{l}
\sup \limits_{ \left\{ x: \, |x -x_{0}| \leq
\frac{T_{0} - \frac{t_{0}(x_{0}) + T(x_{0})}{2}}{\delta_{0}} \right\} } \int_{B(O,1)} (\partial_{s} w(s))^{2} +
|\nabla w(s)|^{2} + w^{2}(s)  \; dy  \\
\leq C \sup \limits_{
\left\{ x: \, |x -x_{0}| \leq \frac{T_{0} - \frac{t_{0}(x_{0}) + T(x_{0})}{2}}{\delta_{0}}  \right\} }
\int_{B \left( 0, \frac{1}{2} \right)}  (\partial_{s} w(s))^{2} + |\nabla w(s)|^{2}  + w^{2}(s)  \; dy \cdot
\end{array}
\label{Eqn:CoverConseqGradw}
\end{equation}
Hence combining this estimate with Lemma \ref{Lem:BoundPoly} we see that (\ref{Eqn:EstKinetw}) holds. This proves Proposition \ref{Prop:BoundPolynws}.

\section{Some intermediate results}
\label{Sec:IntermResults}

In this section we prove some intermediate results (such as Corollary
\ref{Cor:BoundL0} ) which will be used in the next section when we prove Proposition \ref{Prop:BoundFinalWs}. To this end we use some delicate energy, interpolation, and nonlinear estimates. We first use the results of the previous section to derive a polynomial estimate of the $H^{1} \times  L^{2}$- norm of $ \left( w(s), \partial_{s}
w(s) \right)$; this estimate allows to yield a better estimate of the derivative of  $ \mathcal{E}(w) $; finally we design a new Lyapunov functional and we prove that it is bounded.

\subsection{Polynomial pointwise bound of $w(s)$}

First we prove a polynomial pointwise bound of $ \left\| ( w(s),\partial_{s} w(s) ) \right\|_{H^{1} (B(O,1)) \times L^{2} (B(O,1))}$ :

\begin{lem}


There exists $C > 0$  such that for all $s \gg 1$

\begin{equation}
\begin{array}{ll}
\| w(s) \|_{H^{1} \left( B(O,1) \right)} + \| \partial_{s} w(s) \|_{L^{2} \left( B(O,1) \right)} \lesssim  s^{C} \cdot
\end{array}
\label{Eqn:Bound1Ws}
\end{equation}

\label{Lem:PointwiseEstw}
\end{lem}

\begin{proof}

We see from (\ref{Eqn:CoverConseqGradw}) and Proposition \ref{Prop:BoundPolynws} that

\begin{equation}
\begin{array}{l}
\int_{s}^{s+1} \int_{B(O,1)} \left( ( \partial_{\tau} w )^{2} + |\nabla w |^{2} + w^{2} \right) \; \rho dy d \tau  \lesssim  s \log^{1+b}(s) \cdot
\end{array}
\label{Eqn:EstAverw}
\end{equation}
Let $ q  > 1 $ to be chosen later. The mean value theorem shows that there exists $ \sigma(s) \in [s,s+1] $ such that

\begin{equation}
\begin{array}{l}
\int_{B(O,1)}  \left| w (\sigma(s)) \right|^{q} \; \rho dy  =  \int_{s}^{s+1} \int_{B(O,1)}  | w |^{q} \;  dy d \tau
\end{array}
\nonumber
\end{equation}
Hence the combination of the equality above with the fundamental theorem of calculus shows that

\begin{equation}
\begin{array}{l}
\int_{B(O,1)} |w(s)|^{q} \; \rho dy =  \int_{s}^{s+1} \int_{B(O,1)}  | w |^{q} \;  dy d \tau +
\int_{\sigma(s)}^{s}  \frac{d}{d \tau} \int_{B(O,1)} |w(\tau)|^{q} \; \rho dy d \tau \cdot
\end{array}
\label{Eqn:EstLpw}
\end{equation}
We then claim that the following holds: \\
\\
\underline{Claim}: there exists $ C > 0 $ such that the following holds:

\begin{equation}
\begin{array}{l}
N =1, \infty > r \geq 2: \; \int_{B(O,1)} |w(s)|^{r} \; dy  \lesssim s^{C} \\
N > 1: \; \int_{B(O,1)} |w(s)|^{\frac{2N}{N-1}} \; dy  \lesssim s^{C} \cdot
\end{array}
\label{Eqn:LogControlBoundLpw}
\end{equation}

\begin{proof}

Assume that $N=1$ (resp. $N > 1$). Apply (\ref{Eqn:EstLpw}) with $ q = r $ (resp.  $q = \frac{2N}{N-1}$). In order to estimate the first term of the RHS of (\ref{Eqn:EstLpw}), we use the embedding $H^{1} \left( (s,s+1) \times B(O,1) \right) \hookrightarrow L^{r} $  (resp. $x^{q} \lesssim 1 + x^{2^{*}} $ for $x \geq 0$ ). In order to estimate the second term of the RHS of (\ref{Eqn:EstLpw}) we proceed as follows: we first use the composition rule for the derivative; then we apply the estimate $ a b  \leq \frac{a^{2}}{2} + \frac{b^{2}}{2}$ with $ (a,b):= \left( |w(y,\tau)|^{q-1}, |\partial_{\tau} w(y,\tau)| \right) $,  then we use the embedding
$ H^{1} \left( (s,s+1) \times B(O,1) \right) \hookrightarrow L^{2(q-1)} $. Hence using also (\ref{Eqn:EstAverw}) we see that (\ref{Eqn:LogControlBoundLpw}) holds.

\end{proof}

We then claim the following: \\
\\
\underline{Claim}: Let $N > 1$. Then there exists  $\bar{\epsilon} := \bar{\epsilon} (p) > 0 $ for which the following holds: there exist $ \left( \beta, \gamma, C \right) \in \mathbb{R}^{+} \times (0,1) \times \mathbb{R}^{+} $ such that

\begin{equation}
\begin{array}{ll}
p + 1 + \bar{\epsilon} \leq \frac{2N}{N-1}: & \int_{B(0,1)} |w(s)|^{p + 1 + \bar{\epsilon}} \; dy \lesssim s^{C}    \\
p + 1 + \bar{\epsilon} > \frac{2N}{N-1}: &  \int_{B(0,1)} |w(s)|^{p + 1 + \bar{\epsilon}} \; dy  \lesssim \\
&   \left( \int_{B(O,1)} |w(s)|^{\frac{2N}{N-1}} \; dy \right)^{\beta} \left( \int_{B(O,1)} |w(s)|^{2} + |\nabla w(s)|^{2} \; dy \right)^{\gamma}  \cdot
\end{array}
\label{Eqn:BoundLpEpsw}
\end{equation}

\begin{proof}

Indeed let $\bar{\epsilon}> 0$ be such that $ \frac{p + 1 + \bar{\epsilon}}{2} \theta < 1 $ with $\theta$ such that
$ \frac{1}{p + 1 + \bar{\epsilon}} = \frac{1- \theta}{\frac{2N}{N-1}} + \frac{\theta}{\frac{2N}{N-2}} $. Such a choice of $\bar{\epsilon}$
is always possible since $ p < 1 + \frac{4}{N-1} $. If  $ p + 1 + \bar{\epsilon} > \frac{2N}{N-1} $ then the Gagliardo-Nirenberg inequality yields

\begin{equation}
\begin{array}{l}
\| w(s) \|_{L^{p+1 + \bar{\epsilon}}(B(O,1))} \lesssim \| w(s) \|^{1- \theta}_{L^{\frac{2N}{N-1}} (B(O,1))}   \| w (s) \|^{\theta}_{H^{1} (B(O,1))}
\end{array}
\end{equation}
If $ p + 1 + \bar{\epsilon} \leq \frac{2N}{N-1} $ then the H\"older inequality applied to (\ref{Eqn:LogControlBoundLpw}) shows that
(\ref{Eqn:BoundLpEpsw}) holds.

\end{proof}

Elementary estimates show that

\begin{equation}
\begin{array}{l}
a > 0: \\
\begin{array}{ll}
g(\phi w(s)) & \lesssim \log^{a} \left( \log(\phi(s)) \right) \mathbbm{1}_{\{ \phi(s) \geq |w(s)| \}} +
\log^{a} \left( \log(|w(s)|) \right)  \mathbbm{1}_{ \{ |w(s)| \geq \phi(s) \} } \\
& \lesssim  \log^{a}(s) + |w(s)|^{\bar{\epsilon}}
\end{array}
\\
a < 0, \;  \phi w^{2}(s) \gg 1: \, g(\phi w(s)) \lesssim \log^{a} \left( \log(\phi(s)) \right) \\
\end{array}
\label{Eqn:ElemEstgw}
\end{equation}
Moreover we see from (\ref{Eqn:EstDecN})  and (\ref{Eqn:EstJw}) that there exists $C > 0$ such that

\begin{equation}
\begin{array}{l}
\int_{B(O,1)} \left( (\partial_{s} w(s))^{2} +  |\nabla w(s)|^{2} (1 -|y|^{2}) + w^{2}(s) \right) \; \rho  dy \\
\lesssim  e^{-\frac{2(p+1)s}{p-1}} (\log(s))^{\frac{2a}{p-1}} \int_{B(O,1)} F(\phi w(s)) \; \rho  dy  + O \left( (\log(s))^{C} \right) \cdot
\end{array}
\label{Eqn:EstKinetPrelim}
\end{equation}
We then estimate $ \int_{B(O,1)} \left( (\partial_{s} w(s))^{2} +  |\nabla w|^{2} (1 -|y|^{2}) + w^{2}(s) \right) \; \rho  dy $ . Using
(\ref{Eqn:Defphi}) and (\ref{Eqn:ApproxPot}) on $ \{ \phi w^{2}(s) \gg
1 \} $, then 
$ F(\phi w(s)) \lesssim  (\phi w(s))^{p+1} \left( g(\phi w(s)) + 1
\right)  $ on\\
$ \{ \phi w^{2}(s) \lesssim 1 \} \subset \{ w(s) \lesssim \phi^{-\frac{1}{2}}(s) \} $,    we get

\begin{equation}
\begin{array}{rl}
&e^{-\frac{2(p+1)s}{p-1}} (\log(s))^{\frac{2a}{p-1}} \int_{B(O,1)}
  F(\phi w(s)) \rho \; dy\\
  \lesssim
&\frac{1}{(\log(s))^{a}} \int_{ B(O,1) \cap \{ \phi w^{2}(s) \gg 1 \} } |w(y,s)|^{p+1} g (\phi w(s)) \; \rho dy + e^{-s}.
\end{array}
\label{Eqn:PrelimBoundPot}
\end{equation}
We see from (\ref{Eqn:LogControlBoundLpw}), (\ref{Eqn:BoundLpEpsw}), (\ref{Eqn:ElemEstgw}), and (\ref{Eqn:PrelimBoundPot}) that there exists $C > 0$ such that

\begin{equation}
\begin{array}{l}
\int_{B(O,1)} \left( (\partial_{s} w)^{2}(s) +  |\nabla w(s)|^{2} (1 - |y|^{2}) + w^{2}(s) \right) \; \rho  dy \\
\leq   s^{C} \left( 1 +  \left( \int_{B(O,1)}  (\partial_{s} w)^{2}(s) + |\nabla w|^{2}(s) + w^{2}(s) \; dy \right)^{\gamma} \right) \cdot
\end{array}
\label{Eqn:BoundGradw}
\end{equation}
Hence, combining (\ref{Eqn:CoverConseqGradw}) with (\ref{Eqn:BoundGradw}) we see that there exists $ C := C(\delta_{0}) > 0 $ such that

\begin{equation}
\begin{array}{l}
\sup \limits_{ \left\{ x: \, |x-x_{0}| \leq \frac{T_{0} -
\frac{T(x_{0})+t_{0}(x_{0})}{2}} {\delta_{0}} \right\}  } \int_{B(O,1)}  (\partial_{s} w(s))^{2} +
|\nabla w(s)|^{2} + w^{2}(s)  \; dy  \leq  s^{C} \cdot
\end{array}
\label{Eqn:LastStepBoundH1Normw}
\end{equation}
Hence there exists $ C > 0 $ such that (\ref{Eqn:Bound1Ws}) holds. \\

\end{proof}

\subsection{Improved estimate of derivative of $\mathcal{E}(w)$}

The lemma below yields a better estimate of the derivative of $\mathcal{E}(w)$  than (\ref{Eqn:EstDerivE}):

\begin{lem}

  There exists a constant $C > 0$ such that for all $ s \gg 1  $  we have:

\begin{equation}
\begin{array}{ll}
\frac{d \mathcal{E}(w)}{ds} & \leq  - \frac{3 \alpha}{2} \int_{B(O,1)} \frac{(\partial_{s} w)^{2}}{1- |y|^{2}} \rho \; dy \\
+&  \frac{C \log \left( \log( \log(s) ) \right)}{s \log^{a+2}(s)} \int_{B(O,1)} |w|^{p+1} g(\phi w) \rho \; dy + \frac{C}{s^{2}} \int_{B(O,1)} |\nabla w|^{2}(1- |y|^{2}) \rho \; dy \\
+& \frac{C}{s^{2}} \int_{B(O,1)} w^{2} \rho \; dy + \frac{C}{s^{\frac{7}{4}}} \cdot
\end{array}
\label{Eqn:SndDerivE}
\end{equation}

\label{Lem:SndDerivE}
\end{lem}

\begin{proof}

We already know from the proof of Lemma \ref{Lem:EstDerivE} that there exists a constant $C > 0$ such that

\begin{equation}
\begin{array}{ll}
\frac{d \mathcal{E}(w)}{ds} & \leq  - \frac{3 \alpha}{2} \int_{B(O,1)} \frac{(\partial_{s} w)^{2}}{1- |y|^{2}} \rho \; dy
+ \frac{C}{s^{2}} \int_{B(O,1)} |\nabla w|^{2}(1- |y|^{2}) \rho \; dy \\
& + \frac{C}{s^{2}} \int_{B(O,1)} w^{2} \rho \; dy + C e^{-s} + X + Y,
\end{array}
\nonumber
\end{equation}
with $X$ and $Y$ defined in (\ref{Eqn:DefXY}). We have $ X + Y  = \chi_{1}(s) + \chi_{2}(s) $ with

\[
  \begin{array}{rl}
    \chi_{1}(s) := \frac{a}{(p+1) s \log^{a+1}(s) }  \int_{B(O,1)}
  &|w|^{p+1} \log^{a} \left( \log ( 10 + \phi^{2} w^{2} ) \right)\\
&\times    \left( 1 -  \frac{4 s \log(s)}{ (p-1) \log \left( \log ( 10 + \phi^{2} w^{2} ) \right) \log ( 10+ \phi^{2} w^{2}) }   \right) \rho \; dy,
  \end{array}
  \]
  and
  \[
\begin{array}{rl}
    \chi_{2}(s)   := &
  \frac{2 e^{-\frac{2(p+1)s}{p-1}}}{p-1} \log^{\frac{2a}{p-1}}(s)\\
  &\times\int_{B(O,1)}
\left( (p+1) F_{2}(\phi w) - \frac{a}{s \log(s)} F_{1}(\phi w) - \frac{a}{s \log(s)} F_{2}(\phi w) \right) \rho \; dy \cdot
\end{array}
\]
We divide the region of integration into two regions:\\
$ X_{<} := \left\{ y \in B(O,1): \; \phi(s) w^{2}(y,s) \lesssim 1  \right\} $, and
$ X_{>} := \left\{ y \in B(O,1): \; \phi(s) w^{2}(y,s) \gg 1  \right\} $. \\
\\
We first estimate $\chi_{1}(s)$. \\
\\
Let

\begin{equation}
\begin{array}{ll}
\chi_{1,<}(s) & := \frac{a}{(p+1) s \log^{a+1}(s)} \times \\
& \int_{X_{<}}  |w|^{p+1} \log^{a} \left( \log ( 10 + \phi^{2} w^{2} ) \right)
\left( 1 -  \frac{4 s \log(s)}{ (p-1) \log \left( \log ( 10 + \phi^{2} w^{2} ) \right) \log ( 10+ \phi^{2} w^{2}) }   \right) \rho \; dy, \; \text{and} \\
& \\
\chi_{1,>}(s) & :=  \frac{a}{(p+1) s \log^{a+1}(s)}  \times \\
& \int_{X_{>}} |w|^{p+1} \log^{a} \left( \log ( 10 + \phi^{2} w^{2} ) \right) \left( 1 -  \frac{4 s \log(s)}{ (p-1) \log \left( \log ( 10 + \phi^{2} w^{2} ) \right) \log ( 10+ \phi^{2} w^{2}) }   \right) \rho \; dy \cdot
\end{array}
\nonumber
\end{equation}

We first control $\chi_{1,<}(s)$. Clearly

\begin{equation}
\begin{array}{l}
|w|^{p+1} \log^{a} \left( \log (  10 +  \phi^{2} w^{2}) \right) \lesssim \phi^{-\frac{p+1}{2}} \log^{a} \left( \log (  10 +  \phi) \right)
\lesssim e^{-s} \cdot
\end{array}
\label{Eqn:Pointwisew}
\end{equation}
Hence taking into account that $  1 -  \frac{4 s \log(s)}{ (p-1) \log \left( \log ( 10 + \phi^{2} w^{2} ) \right) \log ( 10+ \phi^{2} w^{2}) } \lesssim 1 $ we get

\begin{equation}
\begin{array}{l}
\chi_{1,<}(s) \lesssim e^{-s} \cdot
\end{array}
\nonumber
\end{equation}

We then control $\chi_{1,>}(s)$. We first observe that

\begin{equation}
\begin{array}{l}
\log (10 + \phi^{2} w^{2}) \gtrsim \log(\phi) \gtrsim s, \; \text{and} \\
\log \left( \log (10 + \phi^{2} w^{2}) \right) \gtrsim  \log \left( \log(\phi) \right) \gtrsim \log(s) \cdot
\end{array}
\label{Eqn:LowerBoundDen}
\end{equation}
We have

\begin{equation}
\begin{array}{l}
\log(10 + \phi^{2} w^{2}) = \log(\phi^{2})  +  \log (10 \phi^{-1} + w^{2} ), \; \text{and} \\
\log  \left( \log(10 + \phi^{2} w^{2} ) \right) =  \log \left( \log( \phi^{2} ) \right) + \log \left( 1 + \frac{\log \left(w^{2} + 10 \phi^{-2} \right)}{\log(\phi^{2})} \right) \cdot
\end{array}
\nonumber
\end{equation}
Also $ \log(\phi^{2}) \log \left( \log( \phi^{2} ) \right) = \frac{4s}{p-1} \log(s) + O \left( s \log ( \log (\log(s)) ) \right) $. Hence
$(p-1) \log \left(  \log( 10 + \phi^{2} w^{2} ) \right) \log (10 + \phi^{2} w^{2})  - 4s \log(s) =  X  + Y  + Z + O \left( s \log \left( \log \left( \log(s) \right) \right) \right) $
with

\begin{equation}
\begin{array}{l}
X := (p-1) \log ( w^{2} + 10 \phi^{-2}) \log \left( \log(\phi^{2}) \right), \\
Y  := (p-1) \log(\phi^{2}) \log \left( 1 + \frac{\log (10 \phi^{-2} +  w^{2})}{\log(\phi^{2})} \right), \; \text{and} \\
Z:= (p-1) \log ( w^{2} + 10 \phi^{-2} ) \log \left( 1 + \frac{ \log(w^{2} + 10 \phi^{-2})}{\log(\phi^{2})} \right) \cdot
\end{array}
\nonumber
\end{equation}
Hence using also the elementary inequality $ \log(1 + x) \lesssim  x $ for $x \geq 0$ we get

\begin{equation}
\begin{array}{l}
(p-1) \log \left(  \log( 10 + \phi^{2} w^{2} ) \right) \log (10 + \phi^{2}  w^{2})  - 4s \log(s)  \\
= O \left( s \log ( \log (\log(s)) ) \right) + O \left( \log(s) \log(w^{2} +  10 \phi^{-2} ) \right) \\
+ O \left( \frac{\log^{2}(w^{2} + 10 \phi^{-2})}{s} \right)  \cdot
\end{array}
\nonumber
\end{equation}
Next we claim that the following holds: \\
\\
\underline{Claim}: Let $q \in \{ 1,2 \}$. Then there exists $C > 0 $ such that

\begin{equation}
\begin{array}{l}
\int_{B(O,1)} |w(s)|^{p+1} g(\phi w(s)) \log^{q} \left( 10 + w^{2}(s) \right) \; \rho dy \lesssim  s^{\frac{1}{8}}
+ \int_{B(O,1)} |w(s)|^{p+1} g(\phi w(s)) \; \rho dy
\end{array}
\label{Eqn:BoundPotwLogCorr}
\end{equation}

\begin{proof}

Let $ 0 < \epsilon \ll 1$ and $ 0 < \epsilon^{'} \ll \epsilon $ such that all the statements in the proof are correct. From
$ \log (10 + x^{2}) \lesssim 1 + |x|^{\epsilon^{'}}$ we get

\begin{equation}
\begin{array}{l}
\int_{B(O,1)} |w(s)|^{p+1} g( \phi w(s) ) \log^{q} \left( 10 + w^{2}(s) \right)  \; \rho dy \\
\lesssim \int_{B(O,1)} |w(s)|^{p+1} g( \phi w(s) ) \; \rho dy + \int_{B(O,1)} |w(s)|^{p + 1 + \epsilon'} g(\phi w(s)) \; \rho dy
\end{array}
\nonumber
\end{equation}
Moreover the interpolation inequality w.r.t the measure $ d \mu  := g (\phi w(s)) \rho dy $  shows that there exists $ 0 <  \theta \ll 1 $ such that

\begin{equation}
\begin{array}{l}
\int_{B(O,1)} |w(s)|^{p + 1 + \epsilon'} g(\phi w(s)) \; \rho dy \\
\lesssim   \left( \int_{B(O,1)} |w(s)|^{p+1} g(\phi w(s)) \; \rho dy \right)^{1- \theta}  \left( \int_{B(O,1)} |w(s)|^{p+1 + \epsilon} g(\phi w(s)) \; \rho dy \right)^{\theta}
\end{array}
\nonumber
\end{equation}
We then estimate $g(\phi w)$. If $ a > 0 $ then $ g( \phi w (s) ) \lesssim  \log^{a}(s) + |w(s)|^{\epsilon} $  by proceeding similarly as in (\ref{Eqn:ElemEstgw}).
If $a < 0$ then we clearly have $g(\phi w(s)) \lesssim 1 $. Hence using also the embeddings $ H^{1} ( B(O,1) ) \hookrightarrow L^{p + 1 + \epsilon} (B(O,1)) $
and $ H^{1} ( B(O,1) ) \hookrightarrow L^{p + 1 + 2 \epsilon} (B(O,1)) $, and (\ref{Eqn:Bound1Ws}) we see that there exists $C > 0$ such that

\begin{equation}
\begin{array}{l}
\int_{B(O,1)} |w(s)|^{p+1 + \epsilon} g(\phi w(s)) \; \rho dy \lesssim s^{C} \cdot
\end{array}
\label{Eqn:Boundpowerwplusg}
\end{equation}
Hence we get (\ref{Eqn:BoundPotwLogCorr})  from $s^{C} = s^{C- \frac{1}{8}} s^{\frac{1}{8}} $  (with $C$ defined in (\ref{Eqn:Boundpowerwplusg})) and $ u^{\theta} v^{1 - \theta} \leq \theta u + (1 - \theta) v $.

\end{proof}
Hence by combining this claim with the estimates above we see that
$\chi_{1}(s)$ is bounded by the RHS of (\ref{Eqn:SndDerivE}).

We then estimate $\chi_{2}(s)$.  To this end we write $\chi_{2}(s) = \chi_{2, <}(s) + \chi_{2, >}(s) $ with

\begin{equation}
\begin{array}{l}
\chi_{2, <}(s) :=  \frac{2 e^{-\frac{2(p+1)s}{p-1}}}{p-1} \log^{\frac{2a}{p-1}}(s) \int_{X_{<}}
\left( (p+1) F_{2}(\phi w) - \frac{a}{s \log(s)} F_{1}(\phi w) - \frac{a}{s \log(s)} F_{2}(\phi w) \right) \rho \; dy, \; \text{and} \\
\chi_{2, >}(s) := \frac{2 e^{-\frac{2(p+1)s}{p-1}}}{p-1} \log^{\frac{2a}{p-1}}(s) \int_{X_{>}}
\left( (p+1) F_{2}(\phi w) - \frac{a}{s \log(s)} F_{1}(\phi w) - \frac{a}{s \log(s)} F_{2}(\phi w) \right) \rho \; dy \cdot
\end{array}
\nonumber
\end{equation}
We first control $\chi_{2,<}(s)$. By proceeding similarly as in (\ref{Eqn:ErrorFPhi}) we get

\begin{equation}
\begin{array}{ll}
\chi_{2,<}(s) \lesssim e^{-s} \cdot
\end{array}
\nonumber
\end{equation}

We then control $\chi_{2, >}(s)$. We see from  (\ref{Eqn:DefF1}), (\ref{Eqn:ApproxPot}), and (\ref{Eqn:LowerBoundDen}) that

\begin{equation}
\begin{array}{l}
\chi_{2,>}(s) \lesssim \frac{1}{\log^{a+1}(s) s^{2}} \int_{B(O,1)} |w|^{p+1} g(\phi w) \; \rho  dy \cdot
\end{array}
\nonumber
\end{equation}
Hence $\chi_{2}(s)$ is bounded by the RHS of (\ref{Eqn:SndDerivE}).

\end{proof}

\subsection{Decreasing property of a Lyapunov functional}

Let $u $ be a solution of \eqref{equ}-\eqref{pert2}.
We define $\mathcal{L}_{0}(w)$ in the following fashion:

\begin{equation}
\mathcal{L}_{0}(w): s \rightarrow \mathcal{L}_{0}(w(s)) := \mathcal{E}(w(s))  - \frac{1}{s \log^{\frac{3}{2}}(s)}
\int_{B(O,1)} w(s) \partial_{s} w(s) \rho \; dy \cdot
\label{Eqn:DefL0}
\end{equation}
We prove the following lemma:

\begin{lem}

There exists a constant $ C > 0 $ such that for all $ s \gg 1 $ we have

\begin{equation}
\begin{array}{ll}
\frac{d \mathcal{L}_{0}(w)}{ds} & \leq - \alpha  \int_{B(O,1)}  \frac{(\partial_{s} w)^{2}}{1-|y|^{2}} \rho \; dy
+ \frac{C}{s \log^{\frac{3}{2}}(s)} \mathcal{L}_{0} (w) +  \frac {C}{s^{\frac{7}{4}}} \cdot
\end{array}
\label{Eqn:L0Bound}
\end{equation}

\label{Lem:DerivL0}
\end{lem}

\begin{proof}

Observe that $ \mathcal{L}_{0} (w) = \mathcal{E}(w) + \frac{1}{\log^{\frac{1}{2}}(s)} \mathcal{J} (w) $. Hence

\begin{equation}
\begin{array}{l}
\frac{d \mathcal{L}_{0}(w)}{ds} = \frac{d \mathcal{E}(w)}{ds} - \frac{1}{2 s \log^{\frac{3}{2}}(s)} \mathcal{J}(w)
+ \frac{1}{\log^{\frac{1}{2}}(s)} \frac{d \mathcal{J}(w)}{ds}
\end{array}
\label{Eqn:DerivL0}
\end{equation}
We see from (\ref{Eqn:EstJw}) that $ \mathcal{E}(w) \approx \mathcal{L}_{0}(w) $. This estimate combined with (\ref{Eqn:DerivL0}), (\ref{Eqn:SndDerivE}), (\ref{Eqn:DerivJ}),
and (\ref{Eqn:EstJw}) yields (\ref{Eqn:L0Bound}).

\end{proof}

Let $C$ be the constant defined in Lemma \ref{Lem:DerivL0}. Let $m \in \mathbb{R}^{+}$. We define  the following map:

\begin{equation}
\tilde{\mathcal{L}}_{m}(w):
s \rightarrow \tilde{\mathcal{L}}_{m}(w(s)) :=  e^{\frac{2 C}{\log^{\frac{1}{2}}(s)}} \mathcal{L}_{0} (w(s))  + \frac{m}{s^{\frac{1}{2}}}  \cdot
\label{Eqn:DefL}
\end{equation}
This map has a decreasing property for $m \gg 1$

\begin{lem}

There exist $m \gg 1$ and a constant $ \bar{C} > 0 $ such that for all $s_{2} \geq s_{1} \gg 1 $

\begin{equation}
\begin{array}{l}
\tilde{\mathcal{L}}_{m}(w(s_{2})) - \tilde{\mathcal{L}}_{m} (w(s_{1})) \leq
- \bar{C} \int_{s_{1}}^{s_{2}} \int_{B(O,1)} \frac{(\partial_{\tau} w)^{2}}{1 - |y|^{2}} \rho \; dy \; ds
\end{array}
\label{Eqn:DecL}
\end{equation}

\label{Lem:DecL}
\end{lem}

\begin{proof}

Observe that $f(s) := e^{\frac{2 C}{ \log^{\frac{1}{2}}(s)}}$ satisfies $f^{'}(s) = -\frac{C}{s \log^{\frac{3}{2}}(s)} f(s) $. Hence
by using the operation rules of the derivative we get

\begin{equation}
\begin{array}{l}
\frac{d \tilde{\mathcal{L}}_{m}(w)}{d s} \leq - \alpha e^{\frac{2 C}{ \log^{\frac{1}{2}}(s)}}
 \int_{B(O,1)} \frac{(\partial_{s} w)^{2}}{1 - |y|^{2}} \rho \; dy  + C  e^{\frac{2 C}{ \log^{\frac{1}{2}}(s)}}
\frac{1}{s^{\frac{7}{4}}} - \frac{m}{2} \frac{1}{s^{\frac{3}{2}}}
\end{array}
\nonumber
\end{equation}
Hence we see that (\ref{Eqn:DecL}) holds.

\end{proof}

\subsection{Bounds}

In the corollary below we derive some bounds from Lemma \ref{Lem:DecL}. We have

\begin{cor}

There exists a constant $ C > 0 $ and $ s_{0} \gg 1$ such that
for all $s \geq s_{0} $ we have

\begin{equation}
\begin{array}{l}
-C \leq \mathcal{L}_{0}(w(s)) \leq C \left( \mathcal{L}_{0}(w(s_{0})) + 1 \right) \cdot
\end{array}
\label{Eqn:BoundL0}
\end{equation}
Moreover for all $s_{2} \geq s_{1} \geq s_{0}$ we have

\begin{equation}
\begin{array}{l}
\int_{s_{1}}^{s_{2}} \int_{B(O,1)} (\partial_{\tau} w)^{2} \frac{1}{1- |y|^{2}} \rho dy ds \lesssim  \mathcal{L}_{0}(w(s_{0})) + 1
\end{array}
\label{Eqn:BoundDerivw}
\end{equation}

\label{Cor:BoundL0}
\end{cor}

\begin{proof}

The upper bound in (\ref{Eqn:BoundL0}) is a direct consequence of the decreasing property of $s \rightarrow \tilde{\mathcal{L}}_{m} (w(s)) $ mentioned in  Lemma \ref{Lem:DecL}. It remains to prove the lower bound. By using the same argument to prove that
$\mathcal{N}_{m}(w(s)) \geq 0$ ( see proof of Lemma \ref{Lem:DecreaseN} ) we see that $ \tilde{\mathcal{L}}_{m} (w(s)) \geq 0 $ for all $ s \geq s_{0} $. But
then we see from (\ref{Eqn:DefL}) that the lower bound in (\ref{Eqn:BoundL0}) holds. But then combining (\ref{Eqn:BoundL0}) with (\ref{Eqn:DecL}) we get
(\ref{Eqn:BoundDerivw}).

\end{proof}

\section{Proof of Proposition \ref{Prop:BoundFinalWs} }
\label{Sec:ProofPropBoundFinalWs}
We then prove Proposition \ref{Prop:BoundFinalWs}. To this end we first bound the $s-$ average of the $L^{p+1}\left( B(O,1) \right))$  norm of $w$ on
intervals of size roughly equal to one. More precisely the following lemma holds:

\begin{lem}
Let $s \gg 1$. There exist $C > 0,$  $s_{1} \in [s-1,s] $, and $s_{2} \in [s, s+1]$ such that

\begin{equation}
\begin{array}{l}
\int_{s_{1}}^{s_{2}} \int_{B(O,1)} e^{-\frac{2(p+1)\tau}{p-1}} (\log(\tau))^{\frac{2a}{p-1}} F(\phi w) \; \rho dy d \tau \leq C \cdot
\end{array}
\label{Eqn:BoundAverPotNrj}
\end{equation}

\label{Lem:BoundAvPot}
\end{lem}

The proof is given in Subsection \ref{Subsec:ProofLemBoundAvPot}. \\
With this lemma in mind, we can now prove Proposition \ref{Prop:BoundFinalWs}. \\
\vphantom \\
We see from (\ref{Eqn:Ident0}), (\ref{Eqn:BoundAverPotNrj}), (\ref{Eqn:DecL}), (\ref{Eqn:EstJw}), and
$ |\nabla w(s,y)|^{2} (1 - |y|^{2}) \leq |\nabla w(s,y)|^{2} - |\nabla w(s,y) \cdot y|^{2} $ that

\begin{equation}
\begin{array}{l}
\int_{s}^{s+1} \int_{B(O,1)} \left( (\partial_{s} w )^{2} + |\nabla w|^{2} ( 1 - |y|^{2} ) + w^{2} \right) \; \rho dy d \tau \lesssim 1 \cdot
\end{array}
\label{Eqn:BoundKinetwAver}
\end{equation}
Next we follow again the steps from (\ref{Eqn:EstAverw}) to (\ref{Eqn:LastStepBoundH1Normw}),  except that we use (\ref{Eqn:BoundL0}) instead
of (\ref{Eqn:EstDecN}) to get

\begin{equation}
\begin{array}{l}
\int_{B(O,1)} \left( (\partial_{s} w(s))^{2} +  |\nabla w(s)|^{2} (1 -|y|^{2}) + w^{2}(s) \right) \; \rho  dy \\
\lesssim  e^{-\frac{2(p+1)s}{p-1}} (\log(s))^{\frac{2a}{p-1}} \int_{B(O,1)} F(\phi w(s)) \; \rho  dy  + O(1)
\end{array}
\nonumber
\end{equation}
instead of (\ref{Eqn:EstKinetPrelim}). Observe that the bounds in the process do no longer depend on $s$ since the bound in
(\ref{Eqn:BoundKinetwAver}) does no longer depend on $s$. \\
\\
Hence (\ref{Eqn:BoundFinalWs}) holds.

\subsection{Proof of Lemma \ref{Lem:BoundAvPot}}
\label{Subsec:ProofLemBoundAvPot}

Let $s_{1} \in [s-1,s]$ and $s_{2} \in [s+1,s+2]$ to be chosen. We first state some identities. Integrating $\mathcal{L}_{0}(w(\tau))$ from $s_{1}$ to $s_{2}$ we get from (\ref{Eqn:DefL0}) and (\ref{Eqn:Energyw})

\begin{equation}
\begin{array}{ll}
\int_{s_{1}}^{s_{2}} \mathcal{L}_{0} (w) \; d \tau & =
\int_{s_{1}}^{s_{2}} \int_{B(O,1)}
\left( \frac{1}{2} (\partial_{\tau} w)^{2} + \frac{p+1}{(p-1)^{2}} w^{2} - e^{-\frac{2(p+1) \tau}{p-1}}
(\log(\tau))^{\frac{2a}{p-1}} F(\phi w) \right) \; \rho  dy d \tau \\
& + \frac{1}{2}  \int_{s_{1}}^{s_{2}} \int_{B(O,1)} |\nabla w|^{2} ( 1 - |y|^{2})  \;  \rho  dy \ d \tau
- \int_{s_{1}}^{s_{2}} \frac{1}{\tau \log^{\frac{3}{2}}(\tau)} \int_{B(0,1)} w \partial_{\tau} w \;  \rho dy  d \tau  \cdot
\end{array}
\label{Eqn:Ident0}
\end{equation}
We then multiply (\ref{Eqn:Wavew}) by $w \rho$ and we integrate on $B(O,1) \times [s_{1},s_{2}]$. We get from (\ref{Eqn:ValAlpha}) and some integrations
by parts

\begin{equation}
\begin{array}{l}
\left[ \int_{B(O,1)} \left( w \partial_{\tau} w +  \left( \frac{p+3}{2(p-1)} - N \right) w^{2} \right) \; \rho dy  \right]_{s_{1}}^{s_{2}}  = \\
\\
\int_{s_{1}}^{s_{2}} \int_{B(O,1)} (\partial_{\tau} w)^{2} \; \rho  dy  d \tau
- \int_{s_{1}}^{s_{2}} \int_{B(O,1)}  \left( |\nabla w|^{2}  - |y \cdot \nabla w|^{2} \right) \; \rho  dy  d \tau \\
\\
- \frac{2(p+1)}{(p-1)^{2}} \int_{s_{1}}^{s_{2}} \int_{B(O,1)} w^{2} \; \rho dy  d \tau
+ \int_{s_{1}}^{s_{2}} \int_{B(O,1)} e^{-\frac{2p \tau}{p-1}} (\log(\tau))^{\frac{a}{p-1}} w f(\phi w) \; \rho d y d \tau \\
\\
- 4 \alpha \int_{s_{1}}^{s_{2}} \int_{B(O,1)} w \partial_{\tau} w \frac{|y|^{2}}{1 - |y|^{2}} \; \rho dy d \tau
+ 2 \int_{s_{1}}^{s_{2}} \int_{B(O,1)} (\nabla w \cdot y) \partial_{\tau} w \; \rho  dy  d \tau \\
\\
+ \frac{2a}{p-1} \int_{s_{1}}^{s_{2}} \int_{B(O,1)} \frac{1}{\tau \log(\tau)} (y \cdot \nabla w)  w \; \rho dy d  \tau
+ \int_{s_{1}}^{s_{2}} \int_{B(O,1)} \gamma w^{2} \; \rho dy d \tau \\
\\
+ \frac{2a}{p-1} \int_{s_{1}}^{s_{2}} \int_{B(O,1)} \frac{1}{\tau \log(\tau)} \partial_{\tau} w  w \; \rho  dy d \tau \cdot
\end{array}
\label{Eqn:Ident1}
\end{equation}
Adding (\ref{Eqn:Ident0}) to $ \frac{1}{2} \times \text{(\ref{Eqn:Ident1})}$ we get, by (\ref{Eqn:DecompF}) and (\ref{Eqn:Defphi})

\begin{equation}
\begin{array}{l}
\frac{p-1}{2} \int_{s_{1}}^{s_{2}} \int_{B(O,1)} e^{-\frac{2(p+1) \tau}{p-1}} (\log(\tau))^{\frac{2a}{p-1}} F(\phi w) \; \rho  dy d \tau = \\
\\
\frac{1}{2} \left[  \int_{B(O,1)} \left( w \partial_{\tau} w + \frac{5-p}{2(p-1)} w^{2} \right) \; \rho  dy  \right]_{s_{1}}^{s_{2}}
- \int_{s_{1}}^{s_{2}} \int_{B(O,1)} (\partial_{\tau} w)^{2} \; \rho dy d \tau  \\
\\
+ \int_{s_{1}}^{s_{2}} \mathcal{L}_{0} (w(\tau)) \; d \tau + \frac{4}{p-1} \int_{s_{1}}^{s_{2}} \int_{B(O,1)}
w \partial_{\tau} w \frac{|y|^{2} \rho}{1-|y|^{2}} \; \rho dy d \tau \\
\\
- \int_{s_{1}}^{s_{2}} \int_{B(O,1)} (\nabla w \cdot y) \partial_{\tau} w \; \rho dy  d \tau  + X_{1}+ ... + X_{6}  \cdot
\end{array}
\label{Eqn:CombinEq}
\end{equation}
Here

\begin{equation}
\begin{array}{l}
X_{1} := - \frac{a}{p-1} \int_{s_{1}}^{s_{2}} \int_{B(O,1)} \frac{1}{\tau \log(\tau)} y \cdot \nabla w w \; \rho  d y d \tau, \; X_{2} := - \frac{1}{2} \int_{s_{1}}^{s_{2}} \int_{B(O,1)} \gamma(s) w^{2} \; \rho dy d \tau, \\
\\
X_{3} := -\frac{a}{p-1} \int_{s_{1}}^{s_{2}} \int_{B(O,1)} \frac{1}{\tau \log(\tau)} \partial_{\tau} w  w \; \rho dy d \tau, \;
X_{4} := \int_{s_{1}}^{s_{2}} \frac{1}{\tau \log^{\frac{3}{2}}(\tau)} \int_{B(O,1)} w \partial_{\tau} w \; \rho dy d \tau , \\
\\
X_{5} := \frac{p+1}{2} \int_{s_{1}}^{s_{2}} \int_{B(O,1)} e^{-\frac{2(p+1)}{p-1} \tau} (\log(\tau))^{\frac{2a}{p-1}} F_{1}(\phi w) \; \rho d y d \tau, \; \text{and} \\
\\
X_{6} := \frac{p+1}{2} \int_{s_{1}}^{s_{2}} \int_{B(O,1)} e^{-\frac{2(p+1)}{p-1} \tau} (\log(\tau))^{\frac{2a}{p-1}} F_{2}(\phi w) \; \rho d y d \tau \cdot
\end{array}
\nonumber
\end{equation}
Let $ 1 \gg \epsilon > 0 $ (resp.  $C > 0 $) be a small (resp. large) constant such that all the statements below are true.
Let $ A :=   \epsilon^{-1}  +  \epsilon \int_{s_{1}}^{s_{2}} \int_{B(O,1)} |w|^{p+1}  \; \rho dy d \tau $ and
$ B :=  \epsilon^{-1}  + \epsilon \int_{s_{1}}^{s_{2}} \int_{B(O,1)} e^{- \frac{2(p+1)\tau}{p-1}} (\log(\tau))^{\frac{2a}{p-1}} F( \phi w) \; \rho dy d \tau $. \\
\\
We see from (\ref{Eqn:BoundDerivw}) and similar arguments as those in \cite{MZajm03} (see Section $2$) that

\begin{equation}
\begin{array}{l}
\sup \limits_{s_{1} \leq s \leq s_{2}} \int_{B(O,1)} w^{2}(s) \; \rho dy \lesssim A \cdot
\end{array}
\label{Eqn:BoundMassw}
\end{equation}
Hence using also (\ref{Eqn:Ident0}), (\ref{Eqn:BoundL0}), the inequality $ ab \leq \frac{a^{2}}{2} + \frac{b^{2}}{2}$,
and (\ref{Eqn:BoundDerivw}) we see that the following holds

\begin{equation}
\begin{array}{l}
\int_{s_{1}}^{s_{2}} \int_{B(O,1)} |\nabla w|^{2} (1 - |y|^{2}) \rho \; dy d \tau  \\
\lesssim 1 +  \int_{s_{1}}^{s_{2}} e^{-\frac{2(p+1) \tau}{p-1}} (\log(\tau))^{\frac{2a}{p-1}} \int_{B(O,1)} F (\phi w)  \; \rho dy d \tau \\
+ \frac{1}{s \log^{\frac{3}{2}}(s)} \left(  \int_{s_{1}}^{s_{2}} \int_{B(O,1)} w^{2} \rho \; dy  d \tau
+ \int_{s_{1}}^{s_{2}} \int_{B(O,1)} (\partial_{\tau} w)^{2} \rho \; dy d \tau  \right) \\
\leq C  + C \int_{s_{1}}^{s_{2}} e^{-\frac{2(p+1) \tau}{p-1}} (\log(\tau))^{\frac{2a}{p-1}} \int_{B(O,1)} F(\phi w)  \; \rho dy d \tau + A  \cdot
\end{array}
\label{Eqn:BoundMassGradientw}
\end{equation}
Hence, using also (\ref{Eqn:HardyType})

\begin{equation}
\begin{array}{ll}
|X_{1}| & \lesssim \frac{1}{s \log(s)}
\left(
\int_{s_{1}}^{s_{2}} \int_{B(O,1)}  |\nabla w|^{2} ( 1- |y|^{2}) \rho \; dy d \tau
+ \int_{s_{1}}^{s_{2}} \int_{B(O,1)}  w^{2} \frac{|y|^{2}}{1 - |y|^{2}} \; \rho dy d \tau \right) \\
& \lesssim A  \cdot
\end{array}
\nonumber
\end{equation}
By using (\ref{Eqn:BoundDerivw}), (\ref{Eqn:BoundMassGradientw}), and similar arguments as those in \cite{MZajm03} (see Section $2$) we get

\begin{equation}
\begin{array}{l}
\int_{s_{1}}^{s_{2}} \int_{B(O,1)} |\partial_{\tau} w|  |\nabla w \cdot y| \; \rho dy ds  +
\int_{s_{1}}^{s_{2}} \int_{B(O,1)} |\partial_{\tau} w| |w| \frac{|y|^{2}}{1- |y|^{2}} \; \rho dy ds \lesssim A + B \cdot
\end{array}
\nonumber
\end{equation}
We get from (\ref{Eqn:BoundMassw}) and (\ref{Eqn:BoundDerivw})

\begin{equation}
\begin{array}{l}
| X_{2} | + | X_{3}| + | X_{4} | \lesssim A  \cdot
\end{array}
\nonumber
\end{equation}
We then estimate $|X_{5}|$ and $|X_{6}|$. To this end we proceed similarly as in (\ref{Eqn:ErrorFPhi}) by dividing into the
cases $ \{ \phi w^{2} \gg 1 \} $ and $ \{ \phi w^{2} \lesssim  1 \} $ for all $ \tau \in [s_{1},s_{2}] $ to get

\begin{equation}
\begin{array}{l}
|X_{5}| + |X_{6}| \lesssim e^{-s} + \frac{1}{s \log(s)} \int_{s_{1}}^{s_{2}} \int_{B(O,1)}  e^{-\frac{2(p+1) \tau}{p-1}}
(log(\tau))^{\frac{2a}{p-1}} F(\phi w ) \; \rho  dy  d \tau \lesssim B  \cdot
\end{array}
\nonumber
\end{equation}
Now we prove the following claim: \\
\\
\underline{Claim}: $ A \lesssim  B + 1 $. \\
\\
Indeed by dividing again into the cases $ \{ \phi w^{2} \gg 1 \} $ and $ \{ \phi w^{2} \lesssim  1 \} $
for all $ s \in [s_{1},s_{2}] $ and by using the estimate
$ g (\phi w) \gtrsim \log^{a}(\tau) $ in the region $ \{ \phi w^{2} \gg 1 \}$, we see that

\begin{equation}
\begin{array}{l}
\int_{s_{1}}^{s_{2}} \int_{B(O,1)} |w|^{p+1} \; \rho dy d \tau \lesssim 1 + \int_{s_{1}}^{s_{2}}
\int_{B(O,1) \cap \{ \phi w^{2} \gg 1 \} } \frac{1}{\log^{a}(\tau)} |w|^{p+1} g(\phi w) \; \rho dy d \tau
\end{array}
\nonumber
\end{equation}
From the estimate above  and (\ref{Eqn:ApproxPot}) we see that the claim holds.  \\
\\
Hence we can replace $A$ with $B + 1$ in `` $\lesssim A $ ''  of  all the inequalities above where `` $\lesssim A $ '' appears. \\
\\
Finally we estimate the boundary terms of (\ref{Eqn:CombinEq}) for some particular values of $s_{1}$ and $s_{2}$.
The mean value shows that there exist $s_{1} \in [s-1,s]$  and  $s_{2} \in [s,s+1]$ such that

\begin{equation}
\begin{array}{l}
\int_{s-1}^{s} \int_{B(O,1)} \frac{(\partial_{\tau} w)^{2}}{1- |y|^{2}} \rho \; dy d \tau = \int_{B(O,1)}
\frac{( \partial_{\tau} w)^{2}(s_{1})}{1- |y|^{2}} \; \rho dy \; \text{and} \\
\int_{s}^{s+1} \int_{B(O,1)} \frac{( \partial_{\tau} w)^{2}}{1- |y|^{2}} \rho \; dy d \tau = \int_{B(O,1)} \frac{( \partial_{\tau} w)^{2}(s_{1})}{1- |y|^{2}} \; \rho  dy  \cdot
\end{array}
\nonumber
\end{equation}
Hence using also (\ref{Eqn:BoundDerivw}), the inequality $ab \leq \frac{a^{2}}{2} + \frac{b^{2}}{2}$, and (\ref{Eqn:BoundMassw}), we see that

\begin{equation}
\begin{array}{l}
\left| \left[  \int_{B(O,1)} \left( w \partial_{\tau} w + \frac{5-p}{2(p-1)} w^{2} \right) \rho \; dy  \right]_{s_{1}}^{s_{2}} \right| \lesssim B \cdot
\end{array}
\nonumber
\end{equation}

We can now conclude from all the estimates above that (\ref{Eqn:BoundAverPotNrj}) holds.


\appendix
\section{Estimates of functions}

In this section we prove some pointwise estimates for $F$ and other functions. \\
\\
Let $F_{1}$ be the function defined below

\begin{equation}
\begin{array}{ll}
F_{1}(x) := - \frac{2a}{(p+1)^{2}} |x|^{p+1} \frac{\log^{a-1} \left( \log(10 + x^{2}) \right)}{\log(10+x^{2})} \cdot
\end{array}
\label{Eqn:DefF1}
\end{equation}
Let $F_{2}$ be the function such that

\begin{equation}
\begin{array}{l}
F(x) = \frac{x f(x)}{p+1} + F_{1}(x) + F_{2}(x) \cdot
\end{array}
\label{Eqn:DecompF}
\end{equation}

We prove the following lemma:

\begin{lem}
There exists a large constant $A > 0$ such that if $ |u| \geq A $ then the estimates below hold:

\begin{equation}
\begin{array}{l}
F(u) \approx |u|^{p+1} \log^{a} \left( \log (10 + u^{2}) \right), \; \text{and} \\
\\
F_{2}(u) \lesssim  |u|^{p+1} \frac{ \log^{a-1} \left( \log(10 + u^{2}) \right)}{\log^{2}(10+u^{2})} 
\end{array}
\label{Eqn:ApproxPot}
\end{equation}

\end{lem}

\begin{proof}

Let $u \in \mathbb{R}$. Let $ I :=  \frac{2a}{p+1} \int_{0}^{u} \frac{|x|^{p+1} x}{ (10 + x^{2}) \log(10 + x^{2})}  \log^{a-1} \left( \log (10 + x^{2} ) \right) \; dx $.
An integration by parts shows that

\begin{equation}
\begin{array}{l}
\int_{0}^{u} |x|^{p-1} x \log^{a} \left( \log( 10+ x^{2}) \right) \; dx = \frac{|u|^{p+1}}{p+1} \log^{a} \left( \log(10+ u^{2}) \right) - I \cdot
\end{array}
\nonumber
\end{equation}
Observe that for $ |x| \leq |u|$ we have  $ \left|  \frac{|x|^{p+1} x}{ (10 + x^{2}) \log(10 + x^{2})} \log^{a-1}  \left( \log (10 + x^{2} ) \right) \right|
\lesssim 1 + |x|^{p} \frac{ \log^{a-1}  \left( \log (10 + x^{2} ) \right)}{\log(10 + x^{2})}  \lesssim
1 + |u|^{p} \frac{ \log^{a-1}  \left( \log (10 + u^{2} ) \right)}{\log(10 + u^{2})} $: this follows from the increasing property of
$r \rightarrow r^{p} \frac{ \log^{a-1}  \left( \log (10 + r^{2} ) \right)}{\log(10 + r^{2})} $  for $r \gg 1$. Hence by
integration we see that if $|u| \geq A $ then  $ I \lesssim 1 + |u|^{p+1} \frac{ \log^{a-1}  \left( \log (10 + u^{2} ) \right)}{\log(10 + u^{2})} $ and the first estimate of (\ref{Eqn:ApproxPot}) holds. \\
\\
We turn now our attention to the second estimate of (\ref{Eqn:ApproxPot}). We write $ I = J + K $ with

\begin{equation}
\begin{array}{l}
J := \frac{2a}{p+1} \int_{0}^{u} \frac{|x|^{p-1} x} {\log(10+x^{2})}  \log^{a-1} \left( \log (10 + x^{2} ) \right) \; dx
\end{array}
\nonumber
\end{equation}
We first estimate $J$. An integration by parts shows that there exists a
$\mathcal{C}^{0}([0,u])-$ function $h$ such that $|h(x)| \lesssim \frac{\log^{a-1} \left( \log (10 + x^{2}) \right)}{\log^{2}(10+x^{2})}$ for $|x| \geq 1$
and such that

\begin{equation}
\begin{array}{l}
J = \frac{2a}{(p+1)^{2}} \frac{|x|^{p+1} \log^{a-1} \left( \log(10+x^{2}) \right)}{\log(10 + x^{2})} - \int_{0}^{u} |x|^{p-1} x  h(x) \; dx \\
\end{array}
\nonumber
\end{equation}
Hence if $|u| \geq A$ then

\begin{equation}
\begin{array}{l}
\left| J - \frac{2a}{(p+1)^{2}} \frac{|u|^{p+1} \log^{a-1} \left( \log(10+ u^{2}) \right)}{\log(10 + x^{2})} \right| \lesssim
|u|^{p+1} \frac{ \log^{a-1} \left( \log(10 + u^{2}) \right)}{\log^{2}(10+u^{2})}
\end{array}
\nonumber
\end{equation}
We then write a (rough) estimate of $K$. We have $ K = - \frac{20a}{p+1} \int_{0}^{u} \frac{|x|^{p-1} x}{(10+ x^{2}) \log{(10 +x^{2})}} \log^{a-1} \left( \log (10 + x^{2} ) \right) \; dx $. If $ |x| \leq |u|$ then elementary considerations show that $ \left| \frac{ |x|^{p-1} x }{(10+ x^{2}) \log{(10 +x^{2})}} \log^{a-1} \left( \log (10 + x^{2} ) \right) \right| \lesssim 1 + |u|^{p} \frac{ \log^{a-1}  \left( \log (10 + u^{2} ) \right)}{\log^{2}(10 + u^{2})} $ . Hence by integration if $ |u| \geq A $ then

\begin{equation}
\begin{array}{l}
|K| \lesssim |u|^{p+1} \frac{ \log^{a-1} \left( \log(10 + u^{2}) \right)}{\log^{2}(10+u^{2})}
\end{array}
\nonumber
\end{equation}

\end{proof}

\section{Study of an ODE}

In this section we study the ODE associated to the equation
(\ref{equ}). \\
\\
We consider the following problem

\begin{equation}
\left\{
\begin{array}{l}
v^{''}(t) := |v(t)|^{p-1} v(t) g(t)  \\
v(0) := A > 0 \; \text{and} \; v^{'}(0) := B > 0
\end{array}
\right.
\label{Eqn:ProblemOdeV}
\end{equation}

\begin{lem}
There exists one positive solution to (\ref{Eqn:ProblemOdeV}). Moreover there exists $T < \infty$ such that this solution blows up in finite time $T$ and, moreover, it has the following asymptotic

\begin{equation}
\begin{array}{l}
v(t) \approx (T-t)^{-\frac{2}{p-1}} \log^{-\frac{a}{p-1}}(-\log(T-t)), \; \text{as} \; t \rightarrow T^- \cdot
\end{array}
\label{Eqn:Asympv}
\end{equation}

\label{Lem:OdeAsymp}
\end{lem}

\begin{proof}
The local existence and uniqueness follows from the Picard-Lindel\"of theorem (observe that $y \rightarrow |y|^{p-1} y
\log^{a} \left( \log (10 + y^{2}) \right) $ is locally Lipschitz).  Let $T$ be the maximal positive time of existence of the solution.\\
\\
We claim that $ T < \infty $. Indeed multiplying (\ref{Eqn:ProblemOdeV}) by $v^{'}(t)$ and integrating we see that there exists
a constant $C \in \mathbb{R}$ such that

\begin{equation}
\begin{array}{l}
(v^{'}(t))^{2} = 2F(v(t)) + C
\end{array}
\nonumber
\end{equation}
We claim that $v$ is positive on $[0, \infty)$. If not introducing $ t_{0} := \min \left\{ t \in [0, \infty) : v(t) \leq 0 \right\} $ we see that
for $t < t_{0}$, $ v(t) > 0 $; hence, using also (\ref{Eqn:ProblemOdeV})),  $v^{'}(t) > 0$, which contradicts $v(t_{0}) =0$. Hence
using again  (\ref{Eqn:ProblemOdeV})), we see that $ v^{'}(t) > 0 $ for $t \in [0, \infty)$. Hence

\begin{equation}
\begin{array}{l}
v^{'}(t)= \sqrt{2 F(v(t)) + C} \cdot
\end{array}
\label{Eqn:DerivPrime}
\end{equation}
Hence integrating (\ref{Eqn:DerivPrime}) from $0$ to $t$ we get

\begin{equation}
\int_{0}^{t} \frac{v^{'}(t)}{\sqrt{2 F(v(t)) + C}} \; ds = t \cdot
\label{Eqn:EqualContr}
\end{equation}
We see from (\ref{Eqn:ApproxPot}) that $ \int_{0}^{t} \frac{v^{'}(t)}{\sqrt{2 F(v(t)) + C}} \; ds $ is bounded as $t \rightarrow \infty$. This
contradicts (\ref{Eqn:EqualContr}). Hence $T < \infty$. \\
\\
The blow-up criterion (i.e $ |v(t)| + |v^{'}(t)| \rightarrow \infty $ as $ t \rightarrow T^{-}$), the increasing and positive properties of $v$ and $v^{'}$,
and (\ref{Eqn:DerivPrime}) combined with (\ref{Eqn:ApproxPot}) show that $ v(t) \rightarrow \infty $ as $ t \rightarrow T^{-}$. Integrating (\ref{Eqn:DerivPrime}) from $t$ to $T$ we see that

\begin{equation}
\begin{array}{l}
T - t = \int_{t}^{\infty} \frac{v^{'}(s)}{\sqrt{F(v(s)) +C}} \; ds =  \int_{v(t)}^{\infty} \frac{dy}{\sqrt{F(y) + C}} \cdot
\end{array}
\label{Eqn:EstTmint}
\end{equation}
Hence for $t < T $ close enough to $T$ we get

\begin{equation}
\begin{array}{l}
\int_{v(t)}^{\infty} \frac{dy}{y^{\frac{p+1}{2}+}} \; dy  \leq T - t \leq \int_{v(t)}^{\infty} \frac{dy}{y^{\frac{p+1}{2}-}} \; dy \cdot
\end{array}
\nonumber
\end{equation}
Hence

\begin{equation}
\begin{array}{l}
\frac{1}{(T-t)^{\frac{p-1}{2}+}} \lesssim  v(t) \lesssim \frac{1}{(T - t)^{\frac{p-1}{2}-}} \cdot
\end{array}
\nonumber
\end{equation}
Hence

\begin{equation}
\begin{array}{l}
\log \left( \log \left( 10 + v^{2}(t) \right) \right) \approx \log \left( - \log (T-t) \right) \cdot
\end{array}
\nonumber
\end{equation}
Hence, plugging the estimate above in (\ref{Eqn:EstTmint}) we see that (\ref{Eqn:Asympv}) holds.

\end{proof}

\section{Local well-posedness theory for
\eqref{equ}-\eqref{pert2}}

\label{Sec:LwpH1L2}

We recall the local well-posedness theory of solutions of
\eqref{equ}-\eqref{pert2}
with data  $ \left( u(0),\partial_{t} u(0) \right) := (u_{0},u_{1}) \in H^{1}_{loc} \times L^{2}_{loc} (\mathbb{R}^{N})$.  \\

First we recall the localized energy estimates (\ref{Eqn:LocNrjEst})
(see e.g. Shatah and Struwe \cite{SSnyu98}).
Let $u$ be a solution of $\partial_{tt}
u - \triangle u = F $ with data $\left( u(0),\partial_{t} u(0)
\right):= (u_{0},u_{1})$. Let $t > 0$. Then we get (formally) the
localized energy estimates:

\begin{equation}
\begin{array}{l}
\left\| \left( u(t), \partial_{t} u(t) \right) \right\|_{H^{1} \times L^{2} (|x| < R)} \lesssim  \| (u_{0},u_{1}) \|_{H^{1} (|x| < R + t) \times L^{2} (|x| < R + t|)} +
\| F \|_{L_{t}^{1} L_{x}^{2} \left( (t ^{'}, x): \, t^{'} \in [0,t], |x| < R + t^{'} \right)}
\end{array}
\label{Eqn:LocNrjEst}
\end{equation}
Let $t_{0} > 0$ (resp. $C^{'} > 0 $) be a constant small enough (resp. large enough) such that all the statements below are true.
Let $x_{0} \in \mathbb{R}^{N}$ and $B(x_{0},R)$ be the closed ball with center $x_{0}$ and radius $R$. Let $ K := \left\{ (t,x):  0  \leq  t  <  t_{0}, \, |x - x_{0}| <  t_{0} - t  \right\}$ and
$ \mathcal{X} := \left\{ v: \, v \in \mathcal{C}_{t} H^{1} (K) \cap \mathcal{C}^{1}_{t} L_{x}^{2} (K) \; \text{and} \; \| v \|_{X} \leq  C^{'} \| (u_{0},u_{1}) \|_{H^{1} ( B(x_{0},t_{0}) )  \times L^{2} ( B(x_{0},t_{0})) }   \right) $ be the Banach space endowed with the norm \\
$ \| v \|_{\mathcal{X}} := \sup \limits_{t \in [0,t_{0})} \sup \limits_{ t_{0} - t > R > 0} \max \left( \| v(t) \|_{H^{1}(B(x_{0},R))}, \| \partial_{t} v (t)\|_{L^{2}(B(x_{0},R))} \right) $.  \\

\begin{rem}
We say that $ v \in \mathcal{C}_{t} L_{x}^{2} (K) $ if and only if for all $ \bar{t} \in [ 0, t_{0} ) $ and for all open ball $ B(x_{0},R) \subsetneq
B \left( x_{0}, t_{0} - \bar{t} \right)$  we have $ \lim \limits_{t \rightarrow \bar{t}} \left\| v(t) -  v(\bar{t}) \right\|_{L^{2}( B(x_{0},R))} = 0 $. A similar definition holds for $ u \in \mathcal{C}_{t} H^{1}(K) $.
\end{rem}
Let $\Psi$ be defined by the formula below:
\begin{equation}
\begin{array}{l}
u \in \mathcal{X}   \rightarrow \Psi(u), \; \text{with} \; \Psi(u)(t) :=  \partial_{t} \mathcal{R}(t) * u_{0} + \mathcal{R}(t) * u_{1} + \int_{0}^{t} \mathcal{R}(t-s) * f(u(s)) \; ds
\end{array}
\label{Eqn:DefPsi}
\end{equation}
Here $\mathcal{R}$ denotes the fundamental solution, that is
$\mathcal{R}$  is the solution of
\eqref{equ} with zero right-hand side and
initial data
$ \left( \mathcal{R}(0),\partial_{t} \mathcal{R}(0) \right) := (0,\delta)$ ( Here $ \delta $ is the standard delta function, i.e $\delta(\phi) := \phi(0)$ for every test function
$\phi$ ).  It is well-known (see e.g \cite{FollPDEsBook}) that if $N=1$ then  $ \langle \mathcal{R}(t), \phi \rangle := \frac{1}{2} \int_{-t}^{t} \phi \; dx $ and for $N > 1$

\begin{equation}
\begin{array}{l}
N = 2k, \; k \in \{ 1,.2,... \}: \;  \mathcal{R}(t)  := \frac{1} {1 \times 3 .... \times (N-1)}  \left[ t^{-1} \partial_{t} \right]^{\frac{N-2}{2}} t^{N-2} T_{t}, \; \text{and} \\
N = 2 k + 1, \;  k \in \{ 1, 2,...\}: \;  \mathcal{R}(t) := \frac{1}{1 \times 3... \times (N-2)} \left[ t^{-1} \partial_{t} \right]^{\frac{N-3}{2}} \left[ t^{N-2} \Sigma_{t} \right]
\end{array}
\label{Eqn:FormulaR}
\end{equation}
Here  $ \langle T_{t}, \phi \rangle :=  \frac{1}{\omega_{N+1}} \int_{|y| < 1} \frac{\phi(ty)}{\sqrt{1 - |y|^{2}}} \; dy $ and
$ \langle \Sigma_{t}, \phi \rangle :=  \frac{1}{\omega_{N}} \int_{|y|=1} \phi(ty)  d \sigma(y) $.   In particular  if $N=2$ then
$ \langle R(t), \phi \rangle := \frac{t}{2 \pi} \int_{|y| < 1} \frac{\phi(ty)}{\sqrt{1- |y|^{2}}} \; dy $ and if $N=3$
then  $ \langle R(t), \phi \rangle := \frac{1}{4 \pi}  \int_{S^{2}}  \phi(ty) \; d \sigma(y) $.  \\
Let $(w,w_{1},w_{2}) \in \mathcal{X}^{3}$. \\
From (\ref{Eqn:LocNrjEst}), the embedding $L^{2p} ( B(0,t_{0})) \hookrightarrow H^{1} ( B(0,t_{0})) $, and the estimate $ g(x) \lesssim 1 + |x|^{0+} $ show that

\begin{equation}
\begin{array}{ll}
\| \Psi(w) \|_{L_{t}^{\infty} H^{1}(K) \times L_{t}^{\infty} L_{x}^{2}(K)} & \lesssim \left\| (u_{0},u_{1}) \right\|_{H^{1} ( B(0,t_{0}) ) \times L^{2}  ( B(0,t_{0}))} + \| f(w) \|_{L_{t}^{1} L_{x}^{2}(K)} \\
& \lesssim \left\| (u_{0},u_{1}) \right\|_{H^{1} ( B(0,t_{0}) ) \times L^{2}  ( B(0,t_{0}))} +
t_{0} \left( \| w \|^{p}_{L_{t}^{\infty} L_{x}^{2p} (K)} +
\| w \|^{p+}_{L_{t}^{\infty} L_{x}^{2p+} (K)} \right)  \\
& \leq C^{'} \left\| ( u_{0},u_{1} ) \right\|_{H^{1} (B(0,t_{0})) \times L^{2} (B(0,t_{0}))} \cdot
\end{array}
\label{Eqn:ExpNonlin}
\end{equation}
We then claim that $ \Psi(w) \in  \mathcal{C}_{t} H^{1} (K) \cap
\mathcal{C}^{1}_{t} L_{x}^{2} (K) $. Translating in space and in time
if necessary we may assume
without loss of generality
that $x_{0}=0$ and $\bar{t}=0 $. \\
We first prove that $w_{l}$, solution of the linear wave equation with
data $ \left( u_{l}(0),\partial_{t} u_{l}(0) \right) :=
(u_{0},u_{1})$, lies in $ \mathcal{C}_{t} H^{1} (K) \cap
\mathcal{C}^{1}_{t} L_{x}^{2} (K)$. If $u_{0}$ and $u_{1}$ are smooth
functions (say, in $\mathcal{C}^{\infty}(B(O,R))$), then it is clear from (\ref{Eqn:FormulaR}) and elementary considerations (such as the Minkowski inequality for integrals and the dominated convergence theorem) that $\mathcal{P}(w_{l})$ holds with $ \mathcal{P}(v): \; \lim \limits_{t \rightarrow \bar{t}} \left\| v(t) -  v(\bar{t}) \right\|_{L^{2}( B(O,R))} = 0 $. In the general case $(u_{0},u_{1}) \in H^{1}_{loc} \times L^{2}_{loc}(\mathbb{R}^{N})$,  $\mathcal{P}(u_{l})$ also holds since we can find smooth functions $\left( u_{0,\epsilon}, u_{1,\epsilon} \right)$
within a $ H^{1} (B(O, R+ t_{0})) \times  L^{2}( B(O, R +t_{0}) ) - $ distance of $\epsilon > 0 $ from $(u_{0},u_{1})$
(with $\epsilon$ that can be chosen arbitrary small), then use (\ref{Eqn:LocNrjEst}) to the solution of the linear wave equation with data
$\left( u_{0}- u_{0,\epsilon},u_{1} - u_{1,\epsilon} \right)$, and then use $\mathcal{P} \left( u_{l,\epsilon} \right)$ with $u_{l,\epsilon}$ solution of the linear wave equation with data $ \left( u_{0,\epsilon}, u_{\epsilon} \right)$. \\
Next we consider  $u_{nl}$, solution of  $ \partial_{tt} w - \triangle w  = f(u) $ with data $ \left( u_{nl}(0),\partial_{t} u_{nl}(0) \right) := (0,0)$. From (\ref{Eqn:LocNrjEst}), similar steps as (\ref{Eqn:ExpNonlin}), and the dominated convergence theorem,  we see that $ \mathcal{P} \left( u_{nl} \right)$ holds. \\
The fundamental theorem of calculus and the Holder inequality  show that

\begin{equation}
\begin{array}{ll}
\| \Psi(w_{2}) - \Psi(w_{1}) \|_{X} & \lesssim   \left\| f(w_{2}) - f(w_{1}) \right\|_{L_{t}^{1} L_{x}^{2}(K)} \\
& \lesssim t_{0} \| w_{2} - w_{2} \|_{L_{t}^{\infty} L_{x}^{2p}(K)} \sum \limits_{ \bar{w} \in \{ w_{1}, w_{2} \} }
\left( \| \bar{w} \|^{p-1}_{L_{t}^{\infty} L_{x}^{2p}(K)} + \| \bar{w} \|^{(p-1)+}_{L_{t}^{\infty} L_{x}^{2p+} (K)}  \right) \cdot
\end{array}
\nonumber
\end{equation}
Hence $\Psi$ is contraction and we can apply the fixed point argument to conclude that there exist a unique $u \in \mathcal{X}$ such that $u = \Psi(u)$.

\section{Influence domain, maximal influence domain}
\label{Sec:InflDomain}

In this section we recall the notion of influence domain and maximal influence domain. \\
Recall (see \cite{Apndeta95})
\footnote{There are slight differences between the definition we give
  and the one in
\cite{Apndeta95}.
  We refer to Section $2$ in Chapter $5$ of this book.}
that a set $\Omega$ is an influence domain if it is an open set of $\{ t \geq 0 \}$ and if $ (\bar{x},\bar{t}) \in \Omega $ implies that
$K(\bar{x},\bar{t}) := \left\{ (t,x) :  0 \leq  t <  \bar{t} - |x- \bar{x}| \right\} \subset \Omega $. We define
$ \q D $, the maximal influence domain, to be the union of the influence domains $ \Omega $ such that there exists a solution $ u \in  \mathcal{C}_{t} H^{1} (\Omega) \cap \mathcal{C}^{1}_{t} L_{x}^{2} (\Omega) $ with data $ \left( u(0), \partial_{t} u(0) \right) := (u_{0},u_{1}) \in H^{1}_{loc} \times L^{2}_{loc} (\mathbb{R}^{N}) $ that satisfies $ u = \Psi(u) $.

\begin{rem}
We say that $ v \in \mathcal{C}_{t} L_{x}^{2} (\Omega) $ if and only for all $\bar{t}$ such that $ \{ t = \bar{t} \}  \cap  \Omega  \neq \emptyset $ and for all compact set
$ \mathcal{K} $ such that $ \bar{t} \in \mathcal{K} \subset \{ t = \bar{t} \}  \cap  \Omega $ we have $ \lim \limits_{t \rightarrow \bar{t}} \left\| v(t) - v(\bar{t}) \right\|_{L^{2}(\mathcal{K})} = 0 $. A similar definition holds for $ v \in \mathcal{C}_{t} H^{1}(\mathcal{K}) $.
\end{rem}
We make the elementary observation that if a cone is an influence domain, then a subcone of this cone is also an
influence domain. Hence, using also Appendix \ref{Sec:LwpH1L2} and the definition of
 $ \q D $, we see that there exists $T: = T(x) > 0 $  such that we can write
 $ \q D =  := \left\{ (t,x):  0 \leq t < T(x) \right\}$. Moreover there are two options. Either there exists $x_{0} \in \mathbb{R}^{N}$ such that $T(x_{0}) =  \infty$  and, in this case, $\q D = \left\{ t \geq 0 \right\}$; or the graph $ x \rightarrow T(x)$ is $1-$Lipschitz.

\def\cprime{$'$} \def\cprime{$'$}
\providecommand{\bysame}{\leavevmode\hbox to3em{\hrulefill}\thinspace}
\providecommand{\MR}{\relax\ifhmode\unskip\space\fi MR }
\providecommand{\MRhref}[2]{%
  \href{http://www.ams.org/mathscinet-getitem?mr=#1}{#2}
}
\providecommand{\href}[2]{#2}

\end{document}